\documentclass{article}

\usepackage{amsmath}
\usepackage{geometry}
\usepackage{cleveref}

\usepackage{amssymb}
\usepackage{amsthm}
\usepackage{mathrsfs}
\usepackage{amsfonts}
\usepackage{array}
\usepackage{stmaryrd}
\usepackage[english]{babel}
\usepackage{float}
\newcommand{\enstq}[2]{\left\{#1\mathrel{}\middle|\mathrel{}#2\right\}}
\newcommand{\norm}[1]{\left\|#1\right\|}
\newcommand{\N}{\mathbb{N}}

\newcommand{\R}{\mathbb{R}}

\newcommand{\abs}[1]{\left\lvert #1 \right\rvert}

\newcommand{\isdef}{\mathrel{\mathop:}=}

\usepackage{todonotes}

\let\div\undefined
\DeclareMathOperator{\div}{\textup{div}} 
\makeatletter
\newsavebox{\@brx}
\newcommand{\llangle}[1][]{\savebox{\@brx}{\(\m@th{#1\langle}\)}%
	\mathopen{\copy\@brx\kern-0.5\wd\@brx\usebox{\@brx}}}
\newcommand{\rrangle}[1][]{\savebox{\@brx}{\(\m@th{#1\rangle}\)}%
	\mathclose{\copy\@brx\kern-0.5\wd\@brx\usebox{\@brx}}}
\makeatother

\makeatletter    
\renewcommand*{\vec}[1]{\boldsymbol{#1}}
\makeatother

\newtheorem{theorem}{Theorem}[section]
\newtheorem{corollary}[theorem]{Corollary}
\newtheorem{lemma}[theorem]{Lemma}
\newtheorem{proposition}[theorem]{Proposition}
\theoremstyle{definition}
\newtheorem{definition}[theorem]{Definition}

\theoremstyle{remark}
\newtheorem{remark}[theorem]{Remark}

\numberwithin{equation}{section}

\begin{document}

\title
{Jump-preserving polynomial interpolation in non-manifold polyhedra}


\author{Martin Averseng\thanks{Laboratoire Angevin de Recherche Mathématique (LAREMA), 2 bd Lavoisier, 49000 Angers, martin.averseng@univ-angers.fr}}




\maketitle
\begin{abstract}
	We construct a piecewise-polynomial interpolant $u \mapsto \Pi u$ for functions $u:\Omega \setminus \Gamma \to \R$, where $\Omega \subset \R^d$ is a Lipschitz polyhedron and $\Gamma \subset \Omega$ is a possibly non-manifold $(d-1)$-dimensional hypersurface. This interpolant enjoys approximation properties in relevant Sobolev norms, as well as a set of additional algebraic properties, namely, $\Pi^2 = \Pi$, and $\Pi$ preserves homogeneous boundary values and jumps of its argument on $\Gamma$. As an application, we obtain a bounded discrete right-inverse of the ``jump" operator across $\Gamma$, and an error estimate for a Galerkin scheme to solve a second-order elliptic PDE in $\Omega$ with a prescribed jump across $\Gamma$. 
\end{abstract}
%

\section{Main results} 

For any open set $U \subset \R^d$, $d \geq 2$, we denote by $H^l(U)$ the completion of $C^\infty(U)$ with respect to the norm
\[\norm{u}^2_{H^l(U)} :=  \sum_{\abs{\alpha} \leq l}\norm{D^\alpha u}^2_{L^2(U)}\,.\]
When $U$ is a Lipschitz domain,\footnote{Precise definitions of all technical terms appearing in Sections \ref{sec:mainResult1} to \ref{sec:outline} are given in Section \ref{sec:basic}.} we define $H^s(U)$ for all real $s$ as in \cite{mclean2000strongly}. For a closed set $F \subset U$, let $H^1_{0,F}(U)$ be the closure of $C^\infty_c(U \setminus F)$ in $H^1(U)$. Let $\Omega \subset \R^d$ be a Lipschitz polyhedron and let $\Omega_h$ be a conforming simplicial mesh of $\Omega$. Let $\Gamma \subset \R^d$ be a $(d-1)$-dimensional hypersurface resolved by $\Omega_h$, in the sense that there is a conforming mesh $\Gamma_h$ of $\Gamma$ whose elements are faces of elements of $\Omega_h$. We assume that the elements of $\Omega_h$ intersecting $\Gamma$ do not intersect $\partial \Omega$. We denote by $V^p(\Omega_h;\Gamma)$ the finite-dimensional subspace of $H^1(\Omega \setminus \Gamma)$ consisting of functions which are polynomial of degree $p$ on each mesh element of $\Omega_h$. Note that we do not assume any regularity for the set $\Gamma$. It can be a non-manifold $(d-1)$-dimensional polyhedral surface, such as the one represented in Figure \ref{fig:junction} for $d=3$. We also make no assumption on the uniformity of the mesh $\Omega_h$.

\label{sec:mainResult1}
\begin{theorem}
	\label{thm:alg1}
	There exists a linear operator 
	$$\Pi_h : H^1(\Omega\setminus \Gamma) \to V^p(\Omega_h;\Gamma)$$ 
	satisfying the following properties.
	\begin{itemize}
		\item[(i)] $\Pi_h u_h = u_h$ for all $u_h \in  V^p(\Omega_h;\Gamma)$
		\item[(ii)] $\Pi_h u \in H^1_{0,\Gamma}(\Omega)$ for all $u \in H^1_{0,\Gamma}(\Omega)$,
		\item[(iii)] $\Pi_h u \in H^1(\Omega)$ for all $u \in H^1(\Omega)$.
	\end{itemize} 
\end{theorem}
We call such an operator a {\em jump-aware interpolant on $V^p(\Omega_h;\Gamma)$}. We say that $\Omega_h$ is $\gamma$-shape regular if for every mesh element $K$ of $\Omega_h$, $\frac{h_K}{\rho_K} \leq \gamma$, where $h_K$ is the diameter of $K$ and $\rho_K$ is the radius of its inscribed sphere.

\begin{figure}[H]
	\centering
	\includegraphics[width=0.2\linewidth]{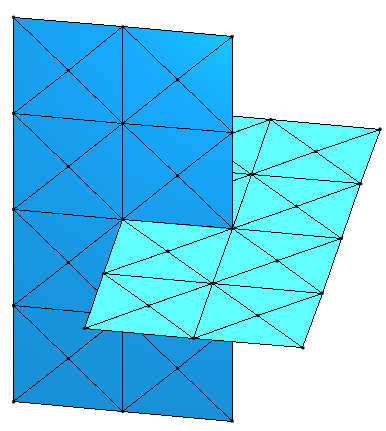}
	\caption{A non-manifold surface $\Gamma \subset \R^3$ and its conforming triangular mesh.}
	\label{fig:junction}
\end{figure}
\begin{theorem}
	\label{thm:cont1}
	Given an integer $p \geq 1$, and real numbers $\frac12 < t \leq p+1$ and $\gamma_0 > 0$, there exists a real number $C(\gamma_0,p,t) >0$ such if $\Omega_h$ is $\gamma_0$-shape regular mesh, there is a jump-aware interpolant $\Pi_h$ with the following continuity properties
		\begin{align}
			\norm{\Pi_h u}_{H^{t}(K)}^2 \leq C(\gamma_0,p,t) \sum_{K' \in \omega_K} \norm{u}_{H^t(K')}^2\,, \label{eq:brokenNorms1}\\
			{\norm{u-\Pi_h u}_{H^{s}(K)}^2}\leq C(\gamma_0,p,t) h_K^{2(t-s)}\sum_{K' \in \omega_K} \abs{u}_{H^t(K')}^2\,,
			\label{eq:brokenNorms2}
		\end{align}
	for any $s \in [0,t]$ and for all $u \in H^t_{\rm pw}(\omega_K) \cap H^1(\Omega \setminus \Gamma)$, where for every mesh element $K$ of $\Omega_h$, $\omega_K \subset \Omega_h$ is the set of mesh elements sharing at least one vertex with $K$. 
\end{theorem}
Here, $H^t_{\rm pw}(\omega_K)$ is the set of functions in $L^2(\Omega)$ whose restriction to any $K' \in \omega_K$ belongs to $H^t(K')$. When $0 \leq s < \frac12$, it is possible to obtain global bounds on the $H^s$ norm of the interpolation error:
\begin{corollary}[Estimate with global norms]
	Under the same assumptions as in \Cref{thm:cont1}, $\Pi_h$ satisfies 
	\[\norm{u - \Pi_h u}_{H^s(\Omega)}^2 \leq \frac{5}{2}C(\gamma_0,p,1) h^{2-2s} \abs{u}^2_{H^1(\Omega \setminus \Gamma)}\]
	for all $0 \leq s < \frac12$ and $u \in H^1(\Omega \setminus \Gamma)$, where $h = \max_K h_K$, and 
	\[\abs{u}^2 = \norm{\nabla u}^2_{L^2(\Omega \setminus \Gamma)}\]
	where $\nabla u$ is the weak gradient of $u$ in $\Omega \setminus \Gamma$ defined in Remark \ref{rem:H1}, Eq. \eqref{eq:weakGrad}.
\end{corollary}
\begin{proof}
	Use \cite[Lemma 4.1.49]{sauter2011bem} and \cite[Lemma 3.33]{mclean2000strongly} and the fact that $H^s_0(\Omega) = H^s(\Omega) =  \widetilde{H}^s(\Omega)$ for these values of $s$. 
\end{proof}

\begin{remark}[The space $H^1(\Omega \setminus \Gamma)$]
	\label{rem:H1}
	Contrarily to other common  notations, the definition taken here, implies that $$H^1(\Omega \setminus \Gamma) \neq H^1(\Omega)$$ 
	in general. For instance, taking $\Gamma = \partial \Omega_0$ where $\overline{\Omega_0} \subset \Omega$ are both bounded Lipschitz polyhedra, the indicator function $\mathbf{1}_{\Omega_0}$ belongs to $H^1(\Omega \setminus \partial \Omega_0)$ but not $H^1(\Omega)$. In this regard, our notation differs e.g. from \cite{veeser2016approximating} (compare \Cref{lem:match} below with Proposition 1 of this reference, keeping in mind that $\Omega$ is a not assumed to be a regular polyhedron therein). By the ``$H = W$" theorem of Meyers and Serrin \cite{meyersHW1964}, $H^1(\Omega \setminus \Gamma)$ as defined here is equal to the set of functions in $L^2(\Omega \setminus \Gamma)$ which possess a weak gradient $p \in L^2(\Omega \setminus \Gamma)^3$, that is
	\begin{equation}
		\label{eq:weakGrad}
		\exists p \in L^2(\Omega \setminus \Gamma)^3 \,:\, \forall \varphi \in \mathcal{D}(\Omega \setminus \Gamma)^3\,,\quad \int_{\Omega \setminus \Gamma} u(x) \div \varphi(x)\,dx = -\int_{\Omega \setminus \Gamma} p(x) \cdot \varphi(x)\,dx\,.
	\end{equation}
	Note that in general, the weak gradient of $u$ is {\em not} equal to the gradient of $u$ in the sense of distributions on $\Omega$. Another common, different definition for $H^1(U)$ is 
	\[H_{\rm restrict}^1(U) := \enstq{f_{|U}}{f \in H^1(\R^d)}\,.\]
	If $U$ is an extension domain, then $H^1_{\rm restrict}(U) = H^1(U)$ with equivalent norms, see \cite[Thm 3.30, (ii)]{mclean2000strongly}, but here, $\Omega \setminus \Gamma$ is not an extension domain in general. Finally, care must be taken when defining a fractional Sobolev space $H^s(\Omega \setminus \Gamma)$. If one uses, for instance, a definition via the Gagliardo semi-norm as in \cite{di2012hitchhiker}, then, perhaps surprisingly, the inclusion $H^1(\Omega \setminus \Gamma) \subset H^s(\Omega \setminus \Gamma)$ for $\frac{1}{2} < s < 1$ {\em does not} hold in general; this is because the Gagliardo semi-norm does not ``see" the set of zero $d$-dimensional measure $\Gamma$, so that $H^s(\Omega \setminus \Gamma) = H^s(\Omega) \nsupseteq H^1(\Omega \setminus \Gamma)$. For the purposes of this work, we will not consider $H^s(U)$ if $U$ is not a Lipschitz domain.
\end{remark}

\subsection{Relation to other results in the literature}

If $\Omega_0$ is a Lipschitz polyhedron, then one can choose a polyhedron $\Omega_0 \Subset \Omega$ and apply \Cref{thm:alg1,thm:cont1} with $\Gamma = \partial \Omega_0$ (i.e. our result applies also if $\Gamma$ is in fact a manifold boundary). It can be seen from the proof below that $\Pi_h u = 0$ on $\Omega \setminus \overline{\Omega_0}$ if $u = 0$ on this set. This way,   
our results extend\footnote{Note, however, that we have restricted our analysis to the $L^2$ setting for simplicity, while these references cover the general $L^q$ setting.} several results of the literature, including \cite[Theorems 2.1, 4.1 and Corollary 4.1]{scott1990finite} (the original construction by Scott and Zhang), by giving the full range of continuity properties in fractional Sobolev norms as in \cite{ciarlet2013analysis}, and giving the estimate in terms of broken norms of $u$ as in \cite{camacho2015pointwise}, (instead of norms over patches). The latter feature, first discovered by Veeser in \cite{veeser2016approximating}, has recently come to attention in several works, see e.g. \cite{gawlik2021local,caetano2022hausdorff,chaumont2021equivalence} and references therein.

\subsection{Background and motivation}

The motivation for this work is the discretization of partial differential equations (PDEs) in complex geometries, in particular, the complement in $\R^d$ of a thin ``screen" or ``crack" $\Gamma$ represented by a $(d-1)$-dimensional surface. These types of geometries occur in a broad range of applications, including antennas \cite{glisson1980simple}, microwaves \cite{chen1970transmission}, satellites \cite{alad2013capacitance}, elasticity \cite{sladek1993nonsingular}, geosciences \cite{lenti2003bem,fan2020high}, water waves \cite{zhao2020iterative} or computer graphics \cite{sharp2020laplacian}. 

This work is especially connected to the development, for such geometries, of boundary element methods (BEM) \cite{mclean2000strongly,sauter2011bem}, also widely known as Method of Moments (MoM) in computational electromagnetics \cite{gibson2021method}. The BEM is usually well-suited for ``obstacle problems", that is, PDEs set in $\mathbb{R}^{d}\setminus \overline{\Omega}$, mainly because  it allows to recast a $d$-dimensional, unbounded problem, into an integral equation on the boundary $\partial \Omega$ of the obstacle. 

In real-life applications with thin obstacles (e.g. metallic plates, screens, fractures), one is led to apply the BEM in the degenerate case where the obstacle becomes an open surface in $\R^3$. This is by now well-understood if the surface is an orientable manifold with boundary, see e.g. \cite{stephan1987boundary,wendland1990hypersingular,buffa2003electric}. The manifoldness makes it possible to regard the fundamental quantities of the BEM -- namely, the ``jumps" $[u]_\Gamma$ of functions $u$ defined in $\R^3\setminus\Gamma$ -- as genuine functions on $\Gamma$. One puts $[u]_\Gamma(x):= \lim_{\varepsilon \to 0^+}u(x + \varepsilon n(x))-u(x-\varepsilon n(x))$, where $n(x)$ is the normal vector at $x$ specified by the orientation. 

When $\Gamma$ is not an orientable manifold, $\Omega$ may be on more than $2$ sides of $\Gamma$ at a given point, so jumps cannot be interpreted as single-valued functions on $\Gamma$. There is nevertheless a more abstract point of view, proposed by Claeys and Hiptmair in \cite{claeys2013integral,claeys2016integral}, which stems from the observation that, when $\Gamma$ is a manifold, the jump $[u]_\Gamma$ of a function $u \in H^1(\Omega \setminus \Gamma)$ vanishes if, and only if, $u \in H^1(\Omega)$. Hence, $[u]_\Gamma$ can be equivalently defined as the equivalence class of $u$ for the relation ``differing by a $H^1(\Omega)$ function", thus extending the definition of $[u]_\Gamma$ to the case where $\Gamma$ is no longer a manifold. It is proved in \cite{claeys2013integral,claeys2016integral} that the PDE can be recast into a uniquely solvable variational problem set in the quotient space associated to this equivalence relation. This has paved the way for the rigorous analysis of non-manifold BEM methods in recent works \cite{claeys2021quotient,averseng2022fractured,cools2022preconditioners,averseng2022ddm}. 

The latest of these works \cite{cools2022preconditioners,averseng2022ddm} have highlighted the need, as a key ingredient for the stability analysis, for a  boundary-element analog to the Scott and Zhang interpolant \cite{scott1990finite}. The latter is a fundamental tool in the analysis of finite-element methods (see e.g. \cite[Thm 1.1]{ainsworth1997posteriori}) and many works are concerned with its study and generalizations, see e.g. \cite{ciarlet2013analysis,ern2017finite,faustmann2019stability} and references therein. For the aforementioned BEM applications, the needed generalization of the Scott-Zhang operator is the one given by \Cref{cor:multitrace} below, see \cite[Proposition 6.9]{averseng2022ddm} and \cite[Theorem 1]{cools2022preconditioners}.

The remainder of this paper is organized as follows. In the next section, we prove a consequence of \Cref{thm:alg1,thm:cont1} regarding polynomial interpolation on non-manifold polyhedral hypersurfaces. We then give two other applications of \Cref{thm:alg1,thm:cont1} in \Cref{sec:application}. An outline of the construction of $\Pi_h$ is given in \Cref{sec:outline}. We then collect definitions and notations in \Cref{sec:basic} and prove \Cref{thm:alg1,thm:cont1} in \Cref{sec:construction}.

\section{Polynomial interpolation on polyhedral hypersurfaces}

\label{sec:corMulti}
From \Cref{thm:alg1,thm:cont1}, one can derive a result concerning piecewise polynomial interpolation of functions defined on the non-manifold hypersurface $\Gamma$. We allow the interpolated functions to be multi-valued on $\Gamma$, as is typically the case when considering restrictions to $\Gamma$ of elements of $H^1(\Omega \setminus \Gamma)$. 

To set a rigorous framework for this problem, we follow \cite{claeys2013integral} by defining the {\em multi-trace space} $\mathbb{H}^{1/2}(\Gamma) := H^1(\Omega \setminus\Gamma) / H^1_{0,\Gamma}(\Omega)$ (equipped with the quotient norm) and let $\gamma_\Gamma:H^1(\Omega \setminus \Gamma) \to \mathbb{H}^{1/2}(\Gamma)$ be the canonical surjection corresponding to this quotient. The polyhedral setup that we have adopted here also allows us to view $\mathbb{H}^{1/2}(\Gamma)$ as a subspace of $L^2(\Gamma) \times L^2(\Gamma)$ in the following way:
\begin{lemma}
	\label{lem:Tr}
 	For each $F \in \Gamma_h$, let $K^+(F)$ and $K^{-}(F)$ be the two elements of $\Omega_h$ that are incident to $F$, (where the $+/-$ labels are chosen arbitrarily). Given $u \in H^1(\Omega \setminus \Gamma)$, define the function $\textup{Tr}^{\pm}(u)$ on $\Gamma$ as 
 	\[(\textup{Tr}^{\pm}u)(x) := \sum_{F \in \mathcal{M}_{\Gamma}} \mathbf{1}_{\{x \in F\}}(\gamma^\pm_{| F} u_{|K^\pm(F)})(x)\] 
 	where $\gamma^\pm_{ F} : H^1(K^\pm(F)) \to L^2(F)$ is the trace operator, and $L^2(F)$ is identified to a subspace of $L^2(\Gamma)$ in the obvious way. Let 
 	\[\textup{Tr}: H^1(\Omega \setminus \Gamma) \to L^2(\Gamma) \times L^2(\Gamma)\]
 	be defined by $\textup{Tr}(u) := (\textup{Tr}^+(u),\textup{Tr}^-(u))$. Then $\textup{Tr}$ is linear and continuous, does not depend on the choice of mesh  $\Omega_h$, and satisfies $\textup{Ker}(\textup{Tr}) = H^1_{0,\Gamma}(\Omega)$.  
\end{lemma}
\begin{proof}
	The independence with respect to the choice of the mesh $\Omega_h$ follows from the fact that $\textup{Tr}$ is invariant under mesh subdivision of $\Omega_h$ and the fact that two meshes of the same polyhedron admit a common subdivision, see \cite[Corollary 1.6]{hudson1969piecewise}. The continuity of $\textup{Tr}$ follows from the trace Theorem, see \cite[Thm 3.37]{mclean2000strongly}. 
	The equality $H^1_{0,\Gamma}(\R^d) = \textup{Ker}(\textup{Tr})$ can be shown as a particular case of the (much more general) result \cite[Theorem 2.2]{swanson1999sobolev} applied to the open set $\Omega := \R^d \setminus \Gamma$. 
\end{proof}
Let $\iota: \mathbb{H}^{1/2}(\Gamma) \to L^2(\Gamma) \times L^2(\Gamma)$ be defined by $\iota (u) \isdef \textup{Tr}(f)$ for any $f \in H^1(\Omega \setminus \Gamma)$ such that $\gamma_\Gamma (f) = u$. By \Cref{lem:Tr}, this is a well-defined injective linear map. We can thus identify $\mathbb{H}^{1/2}(\Gamma)$ to $\iota(\mathbb{H}^{1/2}(\Gamma))$. Under this identification, one has $\gamma_\Gamma = \textup{Tr}$. 

Let $H^{1/2}(\Gamma) := \gamma_\Gamma(H^1(\Omega))$. This space is called the {\em single-trace space}, and is the set of elements of $\mathbb{H}^{1/2}(\Gamma)$ of the form $(u,u)$. We can thus identify it to a subspace of $L^2(\Gamma)$, which coincides (with equivalent norms) with the usual $H^{1/2}(\Gamma)$ when $\Gamma$ is the boundary of a Lipschitz domain (because the trace operator $\gamma_{\partial \Omega}: H^1(\Omega) \to H^{1/2}(\partial \Omega)$ has a continuous right inverse in this case, see \cite[Thm. 3.37]{mclean2000strongly}).  Finally, let $\widetilde{H}^{1/2}(\Gamma) := \mathbb{H}^{1/2}(\Gamma)/H^{1/2}(\Gamma)$ (the {\em jump space}) and let $[\cdot]_\Gamma: \mathbb{H}^{1/2}(\Gamma)\to \widetilde{H}^{1/2}(\Gamma)$ be the corresponding canonical surjection. Define the space of {\em piecewise-polynomial multi-traces} by 
\[\mathbb{V}^p(\Gamma_h) := \textup{Tr}(V^p(\Omega_h;\Gamma))\,.\]
At least if $d \in \{2,3\}$, $p = 1$, and when $d = 3$, if all the vertices of $\Gamma_h$ have ``edge-connected stars" (but this is probably a more general fact), this space only depends on $\Gamma_h$ and not on the tetrahedral mesh $\Omega_h$, because it is equal to the space of ``continuous piecewise linear functions on the inflated mesh", see \cite[Theorem 5.5]{averseng2022fractured}. Following this reference, we define 
\[\Gamma_h^*:= \enstq{(K,F) \in \Omega_h \times \Gamma_h}{K \textup{ is incident to } F}\,,\]
and for $F^* = (K,F) \in \Gamma_h^*$ and $u = (u^+,u^-) \in L^2(\Gamma) \times L^2(\Gamma)$, we write 
\[\norm{u}_{H^s(F^*)} := \norm{u^\pm}_{H^s(F)}\]
if $K = K^\pm(F)$. 
\begin{theorem}[Multi-trace interpolant]
	\label{cor:multitrace}
	For all $\gamma_0 > 0$ and integer $p \geq 1$, there exist a constant $C'(\gamma_0,p) > 0$ such that for all real $s \in [0,\frac12]$, the following holds. If $\Omega_h$ is $\gamma_0$-shape-regular, there exists a linear operator $\Phi_h : \mathbb{H}^{1/2}(\Gamma) \to \mathbb{V}^p(\Gamma_h)$ satisfying 
	\begin{itemize}
		\item[(i)]  $\Phi_h u_h = u_h$ for all $u_h \in   \mathbb{V}^p(\Gamma_h)$. 
		\item[(ii)] $[u]_\Gamma = 0 \Longrightarrow [\Phi_h u]_\Gamma = 0$ for all $u \in \mathbb{H}^{1/2}(\Gamma)$.  
		\item[(iii)] For all $u \in \mathbb{H}^{1/2}(\Gamma)$, 
		\[\norm{\Phi_h u}_{\mathbb{H}^{1/2}(\Gamma)} \leq C'(\gamma_0,p) \norm{u}_{\mathbb{H}^{1/2}(\Gamma)}\,,\]
		\item[(iv)] For all $u \in \mathbb{H}^{1/2}(\Gamma)$
		\[\sum_{F^* \in \Gamma_h^*} h_{F^*}^{2s - 1} \norm{u - \Phi_h u}_{H^s(F^*)}^2 \leq C'(\gamma_0,p) \norm{u}_{\mathbb{H}^{1/2}(\Gamma)}^2\]
	\end{itemize}
	where $h_{F^*} = \max (h_{K^+(F)},h_{K^-(F)})$ with $F^* = (K,F)$.  
\end{theorem}

\begin{remark}
	The operator $\Phi_h$ is analogous the Scott-Zhang interpolant \cite{scott1990finite}, with the multi-trace space $\mathbb{H}^{1/2}(\Gamma)$ playing the role of the volume Sobolev spaces, and with the jump operator $[\cdot]_\Gamma$ playing the role of the trace operator. 
\end{remark}
\begin{proof}[Proof of \Cref{cor:multitrace}]
 	Let $\Pi_h:H^{1}(\Omega \setminus \Gamma) \to V^p(\Omega_h;\Gamma)$ be given by Theorem \ref{thm:cont1}. Let $\mathcal{L}: \mathbb{H}^{1/2}(\Gamma) \to H^{1}(\Omega\setminus \Gamma)$ be the linear isometry satisfying
	\[\norm{\mathcal{L} u}_{H^{1}(\Omega \setminus\Gamma)} = \norm{u}_{\mathbb{H}^{1/2}(\Gamma)} \quad \textup{and} \quad \gamma_\Gamma \mathcal{L} u = \gamma_\Gamma u\]
	for all $u \in \mathbb{H}^{1/2}(\Gamma)$. 
	We define $\Phi_h \isdef \gamma_\Gamma(\Pi_h \circ \mathcal{L})$, and show that properties (i)-(iv) hold. The property (ii) is an immediate consequence of the property (iii) of  \Cref{thm:alg1}. Given $u_h \in \mathbb{V}^p(\Gamma_h)$, let $U_h \in V^p(\Omega_h;\Gamma)$ be such that $u_h = \gamma_\Gamma(U_h)$. On the other hand, let $U = \mathcal{L} u_h$. Noting that $U - U_h \in H^1_{0,\Gamma}(\Omega)$, we deduce that $\Pi_h (U - U_h) \in H^1_{0,\Gamma}(\Omega)$, hence $$\Phi_h u_h - u_h = \gamma_\Gamma(\Pi_h(U - U_h)) = 0\,.$$ 
	This proves (i). Next, notice that the first inequality of \Cref{thm:cont1} with $t = 1$ implies that  
	\[\|\Pi_h u\|_{H^1(\Omega \setminus \Gamma)} \leq C(\gamma_0,p,1) \|u\|_{H^1(\Omega \setminus \Gamma)}\,.\]
	Thus $\|\Phi_h\|_{\mathbb{H}^{1/2}\to \mathbb{H}^{1/2}} \leq  \|\gamma_\Gamma\|_{H^1(\Omega \setminus \Gamma) \to \mathbb{H}^{1/2}}C(\gamma_0,p,1) \|\mathcal{L}\|_{  \mathbb{H}^{1/2}\to H^1(\Omega \setminus \Gamma)} = C(\gamma_0,p,1)$,
	which proves (iii). 
	Finally, given $s \in (0,\frac12]$, we have, for all $F \in \mathcal{M}^*_{\Gamma}$ by the scaled trace theorem (see \Cref{lem:scaledTrace} below)
	\[\begin{array}{rcl}
		\abs{\gamma_{|F}v}^2_{H^s(F)} &\leq& C_{\rm tr} \left(h_K^{-1-2s}\norm{v}_{L^2(K)}^2 + h_K^{1 - 2s}\abs{v}_{H^{1}(K)}^2\right),\\
		\norm{\gamma_{|F}v}^2_{L^2(F)} &\leq& C_{\rm tr}\left(h_K^{-1} \norm{v}_{L^2(K)}^2 + h_K \abs{v}^2_{H^1(K)}\right)\,,
	\end{array}\]
	for all $v \in H^1(K)$, for some constant $C_{\rm tr}$ depending only on $\gamma_0$. Applying these inequalities to $v = \mathcal{L}u - \Pi_h \mathcal{L}u$, we obtain
	\[\begin{split}
		&\sum_{F^* \in \Gamma_h^*} h_{F^*}^{2s - 1} \abs{u - \Phi_h u}_{H^{s}(F^*)}^2\\
		&\qquad\leq C_{\rm tr} \sum_{K \in \Omega_h} h_K^{-2}\norm{\mathcal{L}u - \Pi_h u}^2_{L^2(K)} + \abs{\mathcal{L}u- \Pi_h \mathcal{L}u}_{H^{1}(K)}^2\\
		&\qquad \leq C_{\rm tr}C(\gamma_0,p,1) \sum_{K \in \Omega_h} \abs{\mathcal{L}u}^2_{H^1(K)}\\
		& \qquad\leq C_{\rm tr}C(\gamma_0,p,1) \norm{\mathcal{L}u}^2_{H^1(\Omega \setminus \Gamma)}\\
		&\qquad=C_{\rm tr}  C(\gamma_0,p,1)\norm{u}^2_{\mathbb{H}^{1/2}(\Gamma)}
	\end{split}\] 
	and,
	\[\begin{split}
	&\sum_{F \in \Gamma_h^*} h_{F^*}^{2s - 1} \norm{u - \Phi_h u}_{L^2(F)}^2\\ &\qquad\leq  C_{\rm tr}\sum_{K \in \Omega_h} h_K^{2s-2}\norm{\mathcal{L}u - \Pi_h u}^2_{L^2(K)} + h_K^{2s} \abs{\mathcal{L}u- \Pi_h \mathcal{L}u}_{H^{1}}^2\\
	&\qquad\leq C_{\rm tr}C(\gamma_0,p,1)  \max_{K \in \Omega_h} h_K^{2s}\sum_{K \in \Omega_h} \abs{\mathcal{L}u}^2_{H^1(K)}\\
	&\qquad\leq C_{\rm tr}C(\gamma_0,p,1) \textup{diam}(\Omega)^{2s} \norm{\mathcal{L}u}^2_{H^1(\Omega \setminus \Gamma)}\\
	&\qquad = C_{\rm tr}C(\gamma_0,p,1)\textup{diam}(\Omega)^{2s} \norm{u}^2_{\mathbb{H}^{1/2}(\Gamma)}
\end{split}\]
	where $\textup{diam}(\Omega)$ is the diameter of $\Omega$. Summing these two estimates gives (iv).
\end{proof}

\section{Applications}

\label{sec:application}

We present two applications of \Cref{thm:alg1,thm:cont1}, which parallel the ones given in \cite[Section 5]{scott1990finite} by treating jumps instead of traces. 

\subsection{Bounded discrete right inverse for the jump operator}

With the notation of Section \ref{sec:corMulti}, define the discrete jump space
\[\widetilde{V}^p(\Gamma_h):= [{V}^p(\Omega_h,\Gamma)]_\Gamma = \enstq{[u_h]_\Gamma}{u_h \in V^p_h(\Omega_h;\Gamma)}\,;\]
in other words, $\widetilde{V}^p(\Gamma_h)$ is the set of equivalence classes in $H^1(\Omega \setminus \Gamma)$ of the elements of $V^p(\Omega_h;\Gamma)$ for the relation ``differing by a function in $H^1(\Omega)$".
\begin{corollary}
	\label{discHarmLift}
 	With the same assumptions as in \Cref{thm:cont1}, there exists a linear operator 
	$E_h : \widetilde{H}^{1/2}(\Gamma) \to V^p(\Omega_h;\Gamma)$
	such that 
	\begin{itemize}
		\item[(i)] $[E_h \widetilde{\varphi}_h]_\Gamma = \widetilde{\varphi}_h$ for all $\widetilde{\varphi}_h \in \widetilde{V}^p(\Gamma_h)$, 
		\item[(ii)] for all $\widetilde{\varphi} \in \widetilde{H}^{1/2}(\Gamma)$, 
		\[\norm{E_h \widetilde{\varphi}}_{H^1(\Omega \setminus \Gamma)} \leq C(\gamma_0,p,1)\norm{\widetilde{\varphi}}_{\widetilde{H}^{1/2}(\Gamma)}\,.\]
	\end{itemize}
\end{corollary}
\begin{proof}
	Let $\mathcal{E}: \widetilde{H}^{1/2}(\Gamma) \to {H}^{1}(\Omega \setminus \Gamma)$ be the linear isometry 
	such that
	\[[\mathcal{E} \widetilde{\varphi}]_\Gamma = \widetilde{\varphi} \quad \textup{and}\quad\norm{\mathcal{E} \widetilde{\varphi}}_{{H}^{1}(\Omega \setminus \Gamma)} = \norm{\widetilde{\varphi}}_{\widetilde{H}^{1/2}(\Gamma)} \quad \forall \widetilde{\varphi} \in \widetilde{H}^{1/2}(\Gamma)\,.\]
	We put $E_h:= \Pi_h \circ \mathcal{E}$. Given $\widetilde{\varphi}_h \in \widetilde{V}^p(\Gamma_h)$, let $w_h \in {V}^p(\Omega_h;\Gamma)$ be such that $[w_h]_\Gamma = \widetilde{\varphi}_h$. Then, one has $[\mathcal{E} \widetilde{\varphi}_h - w_h]_\Gamma = 0$, which exactly means that $\mathcal{E}\widetilde{\varphi}_h - w_h \in H^1(\Omega)$. Therefore, the same is true of $\Pi_h(\mathcal{E}\widetilde{\varphi}_h - w_h )$, hence
	$$0 = [\Pi_h (\mathcal{E} \widetilde{\varphi}_h - w_h)]_\Gamma = [\Pi_h (\mathcal{E}\,\widetilde{\varphi}_h) - \Pi_h w_h]_\Gamma = [E_h \widetilde{\varphi}_h]_\Gamma - [w_h]_\Gamma = [E_h \widetilde{\varphi}_h]_\Gamma - \widetilde{\varphi}_h\,.$$
	This proves (i). The property (ii) is immediate.  
\end{proof}

\subsection{Boundary value problems with prescribed jumps}
In \cite[Section 5]{scott1990finite}, Scott and Zhang show that their interpolant provides a ``systematical way for averaging [a prescribed] boundary data" in a boundary value problem. Here we translate this idea to solve an elliptic boundary value problem with a prescribed jump
\begin{equation}
	\label{eq:boundaryJumpProblem}
	\begin{gathered}
		-\sum_{i,j = 1}^3 \frac{\partial}{\partial x_j} \left(\alpha_{ij} \frac{\partial u}{\partial x_i}\right) = 0 \quad \textup{in} \quad \Omega \setminus \Gamma\,, \\
		[u]_\Gamma = [g]_\Gamma\,,\\[0.5em]
		u = 0 \quad \textup{ on } \partial \Omega\,,
	\end{gathered}
\end{equation}
for some given $g \in H^1(\Omega \setminus \Gamma)$ satisfying $g = 0$ on $\partial \Omega$. As in \cite[Section 5]{scott1990finite}, we assume that $\alpha$ is bounded and symmetric positive definite a.e. on $\Omega$. The week solution $u$ of \eqref{eq:boundaryJumpProblem} is defined as the unique element of $H^1(\Omega \setminus \Gamma)$ satisfying $u - g \in H^1_0(\Omega)$ and $a(u-g,v) = 0$ for all $v \in H^1_0(\Omega)$, where 
\[a(u,v) := \int_{\Omega} \sum_{i,j = 1}^3 \alpha_{i,j} \frac{\partial u}{\partial x_i} \frac{\partial v}{\partial x_j}\,dx \,.\]
Define $V^p_{0}(\Omega_h) := V^p(\Omega_h;\Gamma) \cap H^1_0(\Omega)$, and let
\begin{equation}
	V^{p,g}(\Omega_h;\Gamma) := \enstq{v_h \in V^p(\Omega_h;\Gamma)}{v_h - \Pi_h g \in V_{0}^p(\Omega_h)} 
\end{equation}
We can then define an approximation $u_h\in V^{p,g}(\Omega_h;\Gamma)$ by
\[a(u_h,v_h) = 0  \quad \forall v_h \in V^p_0(\Omega_h)\] 
and using the same method of proof as in \cite{scott1990finite}, we obtain
\[\norm{u - u_h}_{H^1(\Omega \setminus \Gamma)} \lesssim \max \enstq{\norm{\alpha_{i,j}}_{L^\infty(\Omega)}}{1 \leq i,j \leq 3} h^{l-1} \abs{u}_{H^l(\Omega \setminus \Gamma)}\,, \,\, 1 \leq l \leq p+1\,,\]
where $h = \max_{K \in \Omega_h} h_K$.

\section{Outline of the construction of $\Pi_h$}
\label{sec:outline}

Replacing $\Gamma$ by $\partial \Omega$ and ignoring property (iii) of Theorem \ref{thm:alg1}, an operator meeting the remaining requirements is given by the Scott and Zhang interpolant $\mathscr{Z}_h$ \cite{scott1990finite}. This operator acts on a function $u$ as 
\begin{equation}
	\label{defSZ}
	\mathscr{Z}_h u \isdef \sum_{i = 1}^N \left(\int_{\sigma_i} \psi_i(x) u(x)\,dx \right) \phi_i =: \sum_{i = 1}^N N_i(u) \phi_i\,.
\end{equation}
Here, $\{\phi_i\}_{i=1}^N$ is the Lagrange nodal basis of $V^p(\Omega_h)$, the space of continuous piecewise polynomial functions of degree $p$ on $\Omega_h$. For Lagrange nodes $\vec x_i$ lying in the interior of some mesh element $K_i \in \Omega_h$, we take $\sigma_i  := K_i$. For $\vec x_i \in \partial K$, $\sigma_i$ is instead a $(d-1)$-simplex, freely chosen among the faces of $\Omega_h$ incident to $\vec x_i$, with the sole restriction that $\sigma_i \subset \partial \Omega$ if $\vec x_i \in \partial \Omega$. Finally, $\psi_i$ is a ``dual Lagrange polynomial" with the property that 
\[\int_{\sigma_i} \psi_i(x) P(x)\,dx = P(\vec x_i) \]
for any polynomial $P$ of degree at most $p$.
It is immediate that $\mathscr{Z}_h u_h = u_h$ for all $u_h \in V^p(\Omega_h)$, and the restriction on $\sigma_i$ also ensure $(\Pi_h u)_{|\partial \Omega} = 0$ when $u_{|\partial \Omega} = 0$. 

One may try to obtain a jump-aware interpolant by adapting the definition in \eqref{defSZ}. The first important change is that for $\vec x_i \in \Gamma$, there generally needs to be more than just one basis function associated to $\vec x_i$, to account for the fact that elements of $V^p(\Omega_h;\Gamma)$ can have several distinct limits at $\vec x_i$ depending on the considered ``side" of $\Gamma$, see Figure \ref{fig:sides}

\begin{figure}
	\centering
	\includegraphics[width=0.5\textwidth]{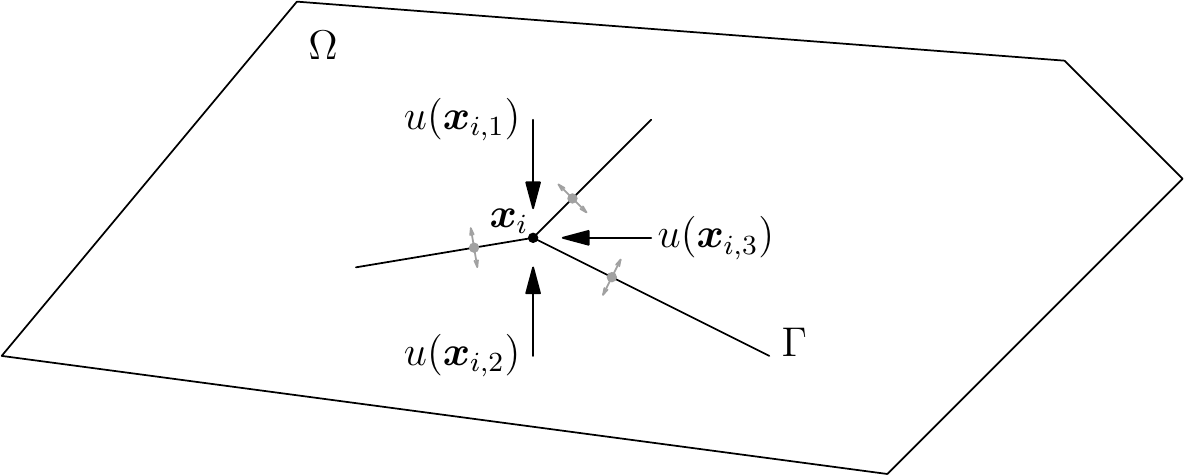}
	\caption{A two-dimensional situation where $\Gamma$ has $q_i = 3$ ``sides" at $\vec x_i$. Functions in $V^p(\Omega_h;\Gamma)$ may admit three distinct limits at $\vec x_i$. The bridges $\{1,2\}$, $\{2,3\}$ and $\{1,3\}$ around $\vec x_i$ are signaled by gray arrows.}
	\label{fig:sides}
\end{figure}
 Labeling the sides of $\Gamma$ around $\vec x_i$ from $1$ to $q_i$, one can construct a nodal basis $\{\phi_{i,j}\}_{1 \leq i \leq N, 1 \leq j \leq q_i}$   with the following properties. Any element $u \in V^p(\Omega_h;\Gamma)$ can be written as
\[u = \sum_{i,j} u(\vec x_{i,j}) \phi_{i,j}\]
where the coefficient $u(\vec x_{i,j})$ is equal to the limit of $u(x)$ as $x$ approaches $\vec x_i$ from side $j$. Furthermore, 
\[\phi_i = \sum_{j = 1}^{q_i} \phi_{i,j} \quad \textup{ and } \quad \sum_{j = 1}^{q_i}\lambda_j \phi_{i,j} \in H^1(\Omega) \iff \lambda_j = \lambda_{j'} \quad  \forall 1 \leq j,j' \leq q_i.\] 
With these definitions, we set
\[\Pi_h u \isdef \sum_{1 \leq i \leq N} \sum_{1 \leq j \leq q_i} N_{i,j}(u)\phi_{i,j}\,,\]
where the $N_{i,j}$ are suitably chosen linear form. For $\vec x_i$ away from $\Gamma$, $q_i = 1$ and $N_{i,1}(u)$ will be defined exactly as $N_i$ in eq.~\eqref{defSZ}; all the difficulty resides in the case where $\vec x_i \in \Gamma$ and $q_i > 1$. In this case, one can attempt to define 
\begin{equation}
	\label{eq:naive}
	N_{i,j}(u) \overset{?}{\isdef} \int_{\sigma_{i,j}} \psi_{i,j}(x) u_{i,j}(x)\,dx \,.
\end{equation}
Here $\sigma_{i,j} \subset \Gamma$ is any face $F \in \mathcal{F}(\Omega_h)$ which contains $\vec x_i$ and lies in the boundary of the $j$-th side of $\Gamma$ around $\vec x_i$. Like above, $\psi_{i,j}$ is the dual Lagrange polynomial associated to $\vec x_i$ on $\sigma_{i,j}$, and $u_{i,j}$ is the restriction of (the polynomial function) $u_{|K_{i,j}}$ to $\sigma_{i,j}$.

The formula \eqref{eq:naive} comes close to satisfying the requirements. It guarantees the conditions $\Pi_h u_h = u_h$ and $(\Pi_h u)_{|\Gamma} = 0$ when $u_{|\Gamma} = 0$. To investigate the condition about jumps, suppose that $u \in H^1(\Omega)$. To fulfill requirement (iii) of \Cref{thm:cont1}, we need to ensure that $N_{i,j}(u) = N_{i,j'}(u)$ for all $1 \leq i \leq N$, $1 \leq j,j' \leq q_i$. If $q_i = 2$, i.e., if there are just two sides of $\Gamma$ around $\vec x_i$ (as for example when $\Gamma$ is a manifold with boundary), this can be arranged by choosing $\sigma_{i,1} = \sigma_{i,2}$ in eq.~\eqref{eq:naive}. However, when there are more than two sides, there is always a pair $j,j'$ such that $\sigma_{i,j} \neq \sigma_{i,j'}$. As soon as this is the case, one can easily find a smooth function $u$ such that that $N_{i,j}(u) \neq N_{i,j'}(u)$, showing that eq.~\eqref{eq:naive} cannot meet the requirements.

A remedy is to define 
\begin{equation}
	\label{modifBridge}
	N_{i,j}(u) = N_i(u) + \sum_{i,k,\ell} \mu_{i,k,\ell}N_{i,k,\ell}
\end{equation} where $N_i$ is as in \eqref{defSZ}, and the $N_{i,k,\ell}$, that we call {\em bridge functions}, are of the form
\begin{equation}
	\label{eq:defNij}
	N_{i,k,\ell}(u) = \int_{\sigma_{i,k,\ell}} \psi_{i,k,\ell}(x) (u_{i,k}(x) - u_{i,\ell}(x))\,dx\,,
\end{equation}
where $\sigma_{i,k,\ell} \subset \Gamma$ is a {\em bridge} that is, a mesh face $F \subset \Gamma$ incident to $\vec x_i$, and located at the interface between the sides $k$ and $\ell$ of $\Gamma$ around $\vec x_i$ (see \Cref{fig:sides}). The key property is that $N_{i,k,\ell}(u) = 0$ when $u \in H^1(\Omega)$, so that bridge functions do not contribute to the value of $N_{i,j}(u)$ in this case. This ensures that $\Pi_h u \in H^1(\Omega)$ whenever $u \in H^1(\Omega)$.  

It remains to make sure that $\Pi_h u_h = u_h$ for all $u_h \in V^p(\Omega_h;\Gamma)$, and it turns out that one can always find a set of coefficients $\mu_{i,k,\ell}$ such that this property holds; as we shall see, this is related to the fact that the graph with nodes the sides $1,\ldots,q_i$ and with edges the bridges $\{k,\ell\}$, is connected.

\section{Definitions and notation}

\label{sec:basic}
We now gather definitions and notations in preparation for the proof of \Cref{thm:alg1} and \Cref{thm:cont1}.

\subsubsection*{Simplices} An {\em $n$-simplex} $S \subset \R^d$, $n \leq d$, is the closed convex hull of $n+1$ affinely independent points in $\R^d$ called its vertices. A {\em $k$-subsimplex} $\sigma$ of $S$ is a $k$-simplex, $k\leq n$, whose vertices are also vertices of $S$. If $k = (n-1)$, $\sigma$ is called a {\em face} of $S$. We denote by $\mathcal{F}(S)$ the set of faces of $S$.  We say that $\sigma$ {\em is incident to} $S$ if it is a subsimplex of $\sigma$, and that two $n$-simplices are {\em adjacent} if they share a face. 

\subsubsection*{Simplicial meshes}
An {\em  $n$-dimensional simplicial mesh} $\mathcal{M}$ is a finite set of $n$-simplices, called {\em (mesh) elements}, such that if $K,K' \in \mathcal{M}$, then $K \cap K'$ is either empty, or equal to a subsimplex of both $K$ and $K'$. The set of {\em faces} of $\mathcal{M}$ is $\mathcal{F}(\mathcal{M}):=\bigcup_{K \in \mathcal{M}}\mathcal{F}(K)$. We say that $\mathcal{M}$ is {\em face-connected} if, for any elements $K,K' \in \mathcal{M}$, there exists a sequence $K_1,\ldots,K_N$ of elements of $\mathcal{M}$ such that $K_1 = K$, $K_N = K'$ and for all $i \in \{1,\ldots,N-1\}$, the elements $K_i$ and $K_{i+1}$ are adjacent. The {\em boundary} $\partial \mathcal{M}$ is the $(n-1)$-dimensional mesh whose elements are the faces of $\mathcal{M}$ incident to only one element of $\mathcal{M}$. 

\subsubsection*{Conforming meshes} For a $n$-dimensional simplicial mesh $\mathcal{M}$ in $\R^d$, we denote by $\abs{\mathcal{M}}$ the closed subset of $\R^d$ defined by 
\[\abs{\mathcal{M}} := \bigcup_{K \in \mathcal{M}} K\,.\]
Given a (closed) set $M \subset \R^d$, we say that $\mathcal{M}$ is a {\em conforming mesh} of $M$ if $\abs{\mathcal{M}} = M$.

\subsubsection*{Regular meshes}
An $n$-dimensional simplicial mesh $\mathcal{M}$ is {\em regular} if it is a conforming mesh of some $n$-dimensional submanifold of $\R^d$ (possibly with boundary). In other words, $\mathcal{M}$ is regular if, for every point of $\abs{\mathcal{M}}$, there is an open set $U_x \subset \R^d$ containing $x$ such that $U_x$ is homeomorphic to either $\R^n$ or $\R^{n-1} \times \R_+$. Recall that for a positive real number $\gamma > 0$, we say that $\mathcal{M}$ is $\gamma$-{\em shape-regular} if
\[\max_{K \in \mathcal{M}}\frac{h_K}{\rho_K} \leq \gamma \,,\] 
where $h_K$ is the diameter of $K$ and $\rho_K$ is the radius of the largest $n$-dimensional closed ball contained in $K$.

We will use the following well-known facts about meshes. We haven't found a complete proof of the first one it in the literature, so we give one for completeness. The second one is proven in \cite[Lemma 11.1.3]{boissonnat1998algorithmic} for connected ``$d$-triangulations", and by \cite[Corollary 1.16]{hudson1969piecewise} and the remark following it, since $\mathcal{M}$ is a regular mesh, $\textup{st}(v)$ (or more precisely, the underlying simplicial complex) is indeed a $d$-triangulation. 

\begin{proposition}[Solid angles in shape-regular meshes]
	\label{prop:solidangles}
	Let $K \subset \R^d$ be a $d$-simplex and $v$ a vertex of $K$. The normalized solid angle $\Omega_v(K)$ of the simplex $K$ from $v$ satisfies
	\[\Omega_v(K) \geq \frac{1}{d} \frac{\rho_{K_i}^d}{{h^d_{K_i}}}\frac{\omega_{d-1}}{\omega_d}\,,\]
	where $\omega_n$ is the volume of the Euclidean unit ball in $\R^n$. 
	Therefore, if $\mathcal{M}$ is a $\gamma$-shape regular $d$-dimensional mesh in $\R^d$, the number $N$ of mesh elements sharing a given vertex is bounded by
	\[N \leq d \gamma^d\frac{\omega_d}{\omega_{d-1}}\,.\]
\end{proposition}
For the definition of the normalized solid angle, refer to \cite{ribando2006measuring}. 
\begin{proof}
	To show the first inequality, we consider the solid $d$-dimensional cone $C$ with apex $v$ and base $S$, where $S$ is a $(d-1)$-dimensional ball centered at the the circumenter $c$ of $K$, with radius $\rho_{K}$, and lying in the plane orthogonal to the line $(vc)$. This is illustrated in Figure \ref{fig:solidangle}. The solid angle spanned by $K$ from $v$ is bounded from below by the solid angle spanned by $C$, for which we can readily prove the claimed estimate. 
	If $K_1,\ldots K_n$ are mesh elements incident to $v$, the normalized solid angle of $\textup{int}(K_1) \cup \ldots \cup \textup{int}(K_n)$ (where $\textup{int}(E)$ stands for the interior of the set $E$) from $v$ is at most $1$, by definition of the normalized solid angle, and equal to $\sum_{i=1}^n \Omega_v(K_i)$ due to the convexity of the simplices $K_i$ and the fact that their interior are pairwise disjoint. This immediately leads to the second inequality.
	\begin{figure}
		\centering
		\includegraphics[width=0.25\textwidth]{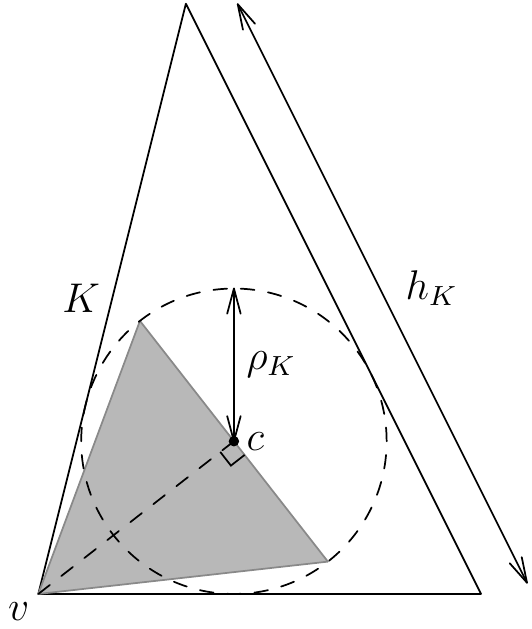}
		\caption{Figure for the proof of \Cref{prop:solidangles}. The cone $C$ is shaded in gray.}
		\label{fig:solidangle}
	\end{figure}
\end{proof}

\begin{proposition}[local face-connectedness in regular meshes]
	\label{prop:faceconnectedness}
	Let $\mathcal{M}$ be a regular $n$-dimensional simplicial mesh with $n \geq 1$, let $v$ be a vertex of some mesh element $K$ and let $\textup{st}(v)$ be the set of mesh elements incident to $v$. Then $\textup{st}(v)$ is face-connected. 
\end{proposition}

\subsubsection*{The sets $\Omega$ and $\Gamma$} For the remainder of this paper, we fix $\Omega \subset \R^d$ a  {\em Lipschitz polyhedron}, that is, a bounded open set with Lipschitz regular boundary (see \cite[Definition 4.4]{evans2018measure}) such that $\overline{\Omega}$ admits a conforming, regular $d$-dimensional mesh $\mathcal{M}_{\Omega}$. We also fix a (possibly non-regular) mesh $\mathcal{M}_{\Gamma}$ such that
\[\mathcal{M}_{\Gamma} \subset \mathcal{F}(\mathcal{M}_{\Omega}) \setminus \partial \mathcal{M}_{\Omega}\,,\]
and let $\Gamma = \abs{\mathcal{M}_\Gamma}$.

\subsubsection*{Continuous function spaces} For an open set $U \subset \R^d$, $C^\infty(U)$ (resp. $C^\infty_c(U)$) is the set of infinitely differentiable functions on $U$ (resp.~which are furthermore compactly supported on $U$), and $C^\infty(\overline{U})$ is the set of restrictions to $U$ of elements of $C^\infty(\R^d)$. For $l \in \N$, let $H^l(U)$ be the completion of $C^\infty(U)$ for the norm 
\[\norm{f}^2_{H^s(U)} \isdef \sum_{\abs{\alpha} \leq l} \norm{D^\alpha f}^2_{L^2(U)}\,.\]
Here, we use the multi-index notation $\alpha = (\alpha_1,\alpha_2,\alpha_3)$, where $\alpha_1,\alpha_2,\alpha_3 \in \N = \{0,1,2,\ldots\}$, $\abs{\alpha} = \sum_{i = 1}^d \alpha_i$ and $D^\alpha = (\partial/\partial_{x_1})^{\alpha_1} (\partial/\partial_{x_2})^{\alpha_2} \ldots (\partial/\partial_{x_d})^{\alpha_d}$. A semi-norm is defined by $\abs{f}^2_{H^l(U)} \isdef \sum_{\abs{\alpha } = l} \abs{D^\alpha f}^2_{L^2(U)}$. If $U$ is a Lipschitz domain, we will also consider the fractional Sobolev space $H^s(U)$ for all $s > 0$, $s \notin \N$, defined as the completion of $C^\infty(U)$ for the norm 
\begin{equation}
	\label{eq:Hs}
	\norm{f}_{H^s(U)}^2 \isdef \norm{f}_{H^l(U)}^2 + \sum_{\abs{\alpha} = l}\abs{D^\alpha f}^2_{H^\mu(U)}\,,
\end{equation}
where $s = l + \mu$, $l \in \N$, $\mu \in (0,1)$, and $\abs{g}_{H^\mu(U)}$ is defined by 
\begin{equation}
	\label{eq:Gagliardo}
	\abs{g}^2_{H^\mu(U)} = \int_{U} \int_{U} \frac{\abs{g(x) - g(y)}^2}{\abs{x - y}^{d + 2\mu}}\,dx\,dy\,.
\end{equation}
Finally, let $H^s_{0,\Gamma}(\Omega)$ be the closure of $C^\infty_c(\Omega \setminus \Gamma)$ in $H^s(\Omega \setminus \Gamma)$. If $K \subset \R^d$ is a $d$-simplex, we write $H^s(K)$ in place of $H^s(\textup{int}(K))$ for all $s \geq 0$.

\subsubsection*{Traces}

If $K \subset \R^d$ is a $d$-simplex, the restriction operator $C^\infty(K) \to L^2(\partial K)$, $u \mapsto u_{|\partial K}$ has a unique continuous linear extension $\gamma_K:H^1(K) \to L^2(\partial K)$, and for all $u \in H^1(K)$ and $\varphi \in C^\infty(\R^d)^3$, 
\begin{equation}
	\label{eq:Green}
	\int_{\Omega} u(x) \div\varphi(x) + \nabla u(x)\cdot \varphi(x)\,dx = \int_{\partial K} (\gamma_K u)(x)\varphi(x) \cdot n_K(x)\,d\mathcal{H}^{d-1}(x)\,,
\end{equation}
where $n_K$ is the unit outer normal vector and $\mathcal{H}^{d-1}$ is the $(d-1)$-dimensional Hausdorff measure, see \cite[Theorem 4.6]{evans2018measure}.

We shall also require fractional Sobolev spaces defined over faces of $d$-simplices. Let $F$ be a $n$-simplex in $\R^d$, with $n \geq 2$ (but possibly $n < d$) we define $H^s(F)$ for $s \in (0,1)$ as the subset of $L^2(F)$ of functions satisfying 
\[\abs{u}^2_{H^s(F)} := \int_{F} \frac{|u(x)-u(y)|^2}{|x - y|^{n + 2s}} d\mathcal{H}^{n}(x)d\mathcal{H}^{n}(y)\,,\]
and let 
\[\norm{u}^2_{H^s(F)}:= \norm{u}^2_{L^2(F)} + \abs{u}^2_{H^s(F)}\,.\]
If $K$ is a $d$-simplex such that and $F \in \mathcal{F}(K)$ is a face of $K$, $s \in (\frac12,1]$, then there exists a constant $C$ (depending on $K$ and $s$) such that the restriction operator $\gamma_{[K,F]}: C^\infty(K) \to L^2(F)$ defined by $\gamma_{[K,F]} u := u_{|F}$
satisfies
\begin{equation}
	\label{eq:traceIneq}
	\norm{\gamma_{[K,F]} u}_{H^{s-\frac12}(F)} \leq C \norm{u}_{H^s(K)}
\end{equation}
for all $u \in H^s(K)$, see \cite[Theorem 3.37]{mclean2000strongly}. The operator $\gamma_{[K,F]}$ admits a unique linear continuous extension from $H^s(K) \to H^{s-\frac12}(F)$, which we denote again by $\gamma_{[K,F]}$. Furthermore, if $U \subset \R^d$ is an open set such that $\textup{int}(K) \subset U$, then for any $u \in H^1(U)$, one has $u_{|K} \in H^1(K)$ and we still denote
\[\gamma_{[K,F]}u := \gamma_{[K,F]} u_{|K}\,.\]

\subsubsection*{Discrete spaces} 
Given a conforming mesh $\Omega_h$ of $\Omega$ (possibly other than $\mathcal{M}_\Omega$ itself), we define
\begin{equation}
	V^p(\Omega_h;\Gamma) \isdef \enstq{u \in {H}^1(\Omega \setminus \Gamma)}{u_{|K} \in \mathbb{P}_p\, \textup{ for all } K \in \Omega_h}\,
\end{equation}
where $\mathbb{P}_p$ is the set of polynomials of degree $p$. We also define
\[\begin{split}
	V^p(\Omega_h) &:= \enstq{u \in H^1(\Omega)}{u_{|K} \in \mathbb{P}_p \textup{ for all } K \in \Omega_h}\\
	&= V^p(\Omega_h;\Gamma) \cap H^1(\Omega)\,.
\end{split}\]
\subsubsection*{Mesh conditions} Given a conforming mesh $\Omega_h$ of $\Omega$, we say that $\Omega_h$ {\em resolves} $\Gamma$ if there is a conforming mesh of $\Gamma$ made of faces of $\Omega_h$, that is 
$$\Gamma = \abs{\Gamma_h} \quad \textup{with} \quad \Gamma_h \subset \mathcal{F}(\Omega_h)\,.$$
Furthermore, $\Omega_h$ {\em strictly encloses} $\Gamma$ if 
$$(K \cap \partial \Omega \neq \emptyset) \Rightarrow (K \cap \Gamma = \emptyset) \qquad\forall K \in \Omega_h\,.$$

\subsubsection*{Lagrange nodes} Let $S \subset \R^d$ be a $n$-simplex, and label its vertices $v_1,\ldots,v_{n+1}$. The {\em Lagrange nodes} of $S$ are the points of S defined by 
\[\mathcal{L}_p(S):= \enstq{\frac{1}{p} \sum_{i = 1}^{n+1} \alpha_i v_i}{\alpha \in \mathbb{N}^{n+1} \textup{ s.t. } |\alpha| = p}\,.\]
Note that if $S'$ is a subsimplex of $S$, then the Lagrange nodes on $S'$ are Lagrange nodes on $S$, i.e. $\mathcal{L}_p(S') \subset \mathcal{L}_p(S)$.
The {\em Lagrange nodes} of a mesh $\mathcal{M}$ are the Lagrange nodes of its elements, i.e.
\[\mathcal{L}_p(\mathcal{M}):= \bigcup_{K \in \mathcal{M}} \mathcal{L}_p(K)\]

\section{Construction of $\Pi_h$}
\label{sec:construction}

In this section we define $\Pi_h$ and prove \Cref{thm:alg1,thm:cont1}. The construction involves a primal and a dual basis, which are defined in \Cref{sec:BasisMulti} and \Cref{sec:dualBasis}, respectively. The definition of $\Pi_h$ and the proof are given in \Cref{sec:defPih}. 

From now on, we fix integers $d\geq 2$ and $p \geq 1$ and real numbers $\gamma_0 >0$ and $t > \frac12$. $\Omega$ and $\Gamma$ are fixed as in the previous section. We write $a \lesssim b$ when there exists $C > 0$ whose value depends only on $\Omega$, $\Gamma$, $d$, $p$, $\gamma_0$ and $t$, such that $a \leq Cb$. We fix a conforming $\gamma_0$-shape-regular mesh $\Omega_h$ of $\Omega$ which resolves and strictly encloses $\Gamma$, and let $\Gamma_h$ be the mesh of $\Gamma$ constituted of faces of $\Omega_h$. 

\subsection{Primal basis of $V^p(\Omega_h;\Gamma)$}

\label{sec:BasisMulti}

It is well-known that a piecewise-polynomial function $u_h$ on $\Omega_h$ belongs to $H^1(\Omega)$ if, and only if, its polynomial values in any two adjacent mesh elements ``match" at the common face. The condition $u_h \in H^1(\Omega \setminus \Gamma)$ similarly involves the continuity of traces at faces of adjacent elements, except when this face is in $\Gamma$.
\begin{lemma}
	\label{lem:match}
	Let $f: \Omega \setminus \Gamma \to \R$ be such that for all $K \in \Omega_h$, there exists $f_K \in H^1(\Omega)$ such that $f_{|K} = f_K$. Then the following properties are equivalent 
	\begin{itemize}
		\item[(i)] $f \in H^1(\Omega \setminus \Gamma)$, 
		\item[(ii)] For any pair of adjacent mesh elements $K,K' \in \Omega_h$ sharing a face $F\notin \Gamma_h$, \[\gamma_{[K,F]}f_K = \gamma_{[K',F]} f_{K'}\,.\]
	\end{itemize} 
\end{lemma} 
\begin{proof}
	Recall that by the $H = W$ theorem \cite{meyersHW1964}, $f \in H^1(\Omega \setminus \Gamma)$ if and only if there exists a square-integrable vector field $g \in L^2(\Omega \setminus \Gamma)^d$ such that 
	\[\int_{\Omega \setminus \Gamma} f(x) \,\textup{div}(\varphi(x))\,dx = -\int_{\Omega \setminus\Gamma} g(x) \cdot \varphi(x)dx \] 
	for all $\varphi(x) \in C^\infty_c(\Omega \setminus \Gamma)^d$. The proof that this happens if and only if (ii) holds is classical.
\end{proof}
One sees from this result that elements of $V^p(\Omega_h;\Gamma)$ may be discontinuous, and especially have non-zero jumps through faces in $\Gamma$. On the other hand, not all polynomial (of degree $p$) jumps are possible through a given face $F$. Indeed, if for some face $F \in \Gamma_h$ and some Lagrange node $\vec x$ on $F$, there is a chain of mesh elements $K_1,\ldots,K_r \in \Omega_h$ with $K_1 = K_{-}(F)$, $K_r = K_+(F)$, such that $K_{i}$ and $K_{i+1}$ share a face $F_i$ not in $\Gamma$ containing $\vec x$ (see Figure \ref{fig:chain}) then the polynomial functions $v_{h_{|K_i}}$ and $v_{h|_{K_{i+1}}}$ agree at $\vec x$. Thus, $v_{h|_{K^+(F)}}(\vec x)= v_{h|_{K^-(F)}}(\vec x)$, hence the jump $\llbracket v_h \rrbracket_F$ vanishes at $\vec x$, for all $v_h \in V^p(\Omega_h;\Gamma)$.

\begin{figure}
	\centering
	\includegraphics[width=0.5\textwidth]{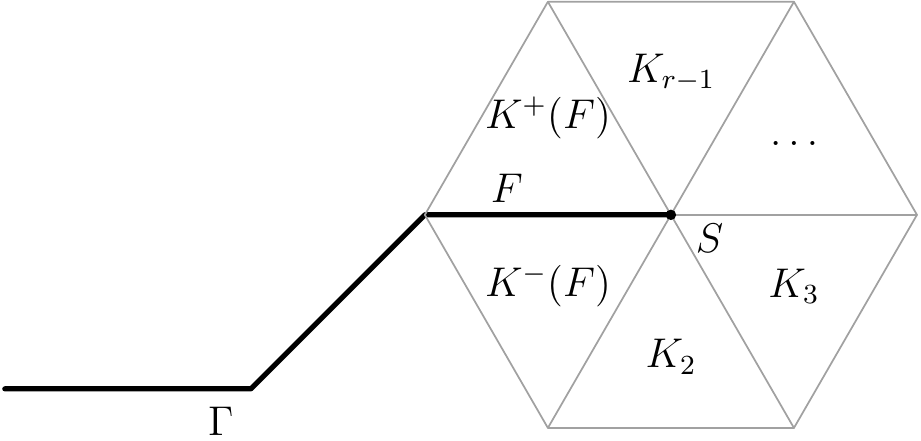}	
	\caption{A face-connected sequence of elements of $\Omega_h$ circling around a subsimplex $S$ and linking $K^-(F)$ with $K^+(F)$.}
	\label{fig:chain}
\end{figure}

To construct a basis of $V^p(\Omega_h;\Gamma)$, it is thus necessary to capture the couplings produced by this kind of chain of elements. In order to do this, denote by $\{\vec{x}_1\,,\ldots\,,\vec{x}_N\}$ the set of Lagrange nodes on $\Omega_h$, with
$$\{\vec x_1,\ldots,\vec x_M\} = \mathcal{L}_p(\Gamma_h) \quad \textup{and}\quad\{\vec x_{M+1}\,,\ldots\,,\vec x_{N}\} = \mathcal{L}_p(\Omega_h) \setminus \mathcal{L}_p(\Gamma_h)\,.$$ 
Moreover, let $\{\phi_i\}_{1 \leq i \leq N}$ be the nodal basis of $V^p(\Omega_h)$, that is, the set of elements of $V^p(\Omega_h)$ defined by
\[\phi_i(\vec{x}_{i'}) = \delta_{i,i'}\,, \quad 1 \leq i,i' \leq N\,.\] 
Define the {\em star} of a subsimplex $\sigma$ of $K \in \mathcal{M}$, as the mesh
$$\textup{st}(\sigma,\mathcal{M}) \isdef \enstq{K \in \mathcal{M}}{K \cap \sigma \neq \emptyset}\,,$$
and let $\textup{st}(\vec x_i) \isdef \textup{st}(\vec x_i,\Omega_h)$ be the set of mesh elements $K \in \Omega_h$ incident to $\vec x_i$. We define the following (unoriented) graph ${G}(\vec{x}_i)$ (see also \cite{averseng2022fractured}):
\begin{itemize}
	\item {\em Nodes:} The mesh elements $K \in \textup{st}(\vec x_i)$,
	\item {\em Edges:} The pairs $\{K,K'\} \subset \textup{st}(\vec x_i)$ such that $K$ and $K'$ share a face $F \notin \Gamma_h$. 
\end{itemize}
Let $q_i$ be the number of connected components of $G(\vec x_i)$. We can then partition the mesh elements $K \in \Omega_h$ incident to $\vec x_i$ as
\[\textup{st}(\vec x_i) = \textup{st}(\vec x_i;1) \cup \ldots \cup  \textup{st}(\vec x_i;q_i) \,.\] 
This is illustrated in Figure \ref{fig:stxi}.
\begin{figure}
	\centering
	\includegraphics[width=0.5\textwidth]{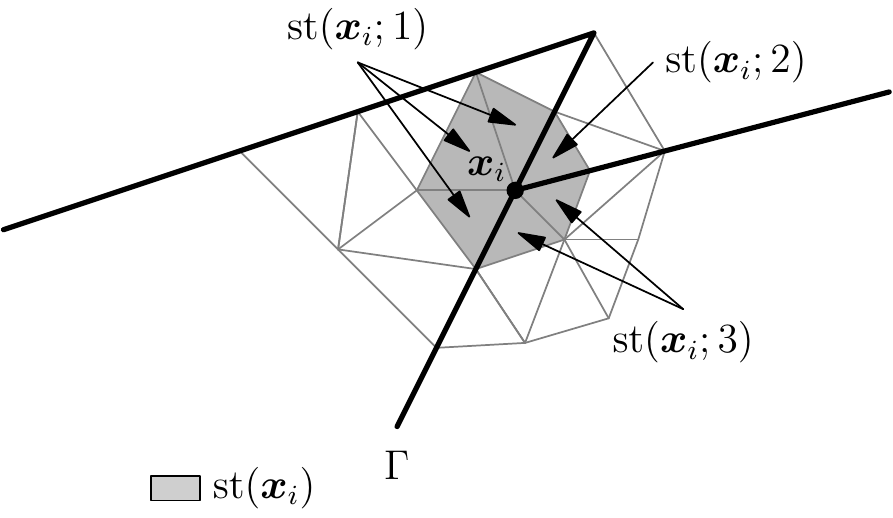}
	\caption{A Lagrange node $\vec x_i$ on the hypersurface $\Gamma$ (solid black line), its star $\textup{st}(\vec x_i)$ (filled in gray), and the components $\textup{st}(\vec x_i;j)$ for $j = 1,2,3$. The gray solid lines represent the boundary of some mesh elements of $\Omega_h$.}
	\label{fig:stxi}
\end{figure}
Each of the disjoint sets $\textup{st}(\vec x_i;j)$ corresponds to one connected component of $G(\vec x_i)$, and is thus a face-connected mesh, with face-connectedness holding through faces not in $\Gamma$.  The quantity $\max_{1 \leq i \leq N}q_i$ is invariant by mesh subdivision of $\Omega_h$ resolving $\Gamma$, and thus 
\begin{equation}
	\label{eq:qmax}
	\max_{1 \leq i \leq N} q_i \lesssim 1\,.
\end{equation}

Let $\mathcal{I} \isdef \enstq{(i,j) \in \N^2}{1 \leq i \leq N\,,\,\, 1 \leq j \leq q_i}$. For $(i,j) \in \mathcal{I}$, we define the {\em split basis function} $\phi_{i,j}: \Omega \setminus \Gamma \to \mathbb{R}$ by 
\[\phi_{i,j}(x) \isdef \begin{cases}\phi_i( x) & \textup{for } x \in \textup{int}(K) \textup{ such that } K \in \textup{st}(\vec x_i;j)\,, \\
	0 & \textup{otherwise}.	
\end{cases}\]
Note that $\sum_{j = 1}^{q_i} \phi_{i,j} = \phi_i$ and by Lemma \ref{lem:match}, $\phi_{i,j} \in V^p(\Omega_h;\Gamma)$ for all $(i,j) \in \mathcal{I}$. 
\begin{lemma}
	\label{basisVh}
	The family $\{\phi_{i,j}\}_{(i,j) \in \mathcal{I}}$ is a basis of $V^p(\Omega_h;\Gamma)$.
\end{lemma}
\begin{proof}
	Suppose that there exist coefficients $\lambda_{i,j}$ such that 
	\begin{equation}
		\label{linDepPhiIJ}
		\forall x \in \Omega \setminus \Gamma\,, \quad \sum_{(i,j) \in \mathcal{H}(\Omega)} \lambda_{i,j} \phi_{i,j}({x}) = 0\,.
	\end{equation}
	Fix $(i_0,j_0) \in \mathcal{H}(\Omega)$ and let $K \in \textup{st}(\vec x_{i_0};j_0)$. Let $( y_{n})_{n \in \N}$ be a sequence of points in the interior of $K$, such that 
	\[\lim_{n \to \infty} y_n = \vec x_{i_0}\,.\]
	It is easy to show that $\lim_{n \to \infty} \phi_{i,j}(y_n) \to \delta_{i,i_0} \delta_{j,j_0}$. Therefore, applying eq.~\eqref{linDepPhiIJ} to $y_n$ and passing to the limit, we conclude that $\lambda_{i_0,j_0} = 0$. This proves that the functions $\phi_{i,j}$ are linearly independent. 
	
	Next, we consider $u_h \in V^p(\Omega_h;\Gamma)$. For each $(i,j) \in \mathcal{H}(\Omega)$, we fix a mesh element $K_{i,j} \in \textup{st}(\vec x_i;j)$ and a sequence $(y^{i,j}_{n})_{n \in \N}$ converging to $\vec x_i$ from $K_{i,j}$, as above. Let $\lambda_{i,j} := \lim_{n \to \infty} u_h(y^{i,j}_n)$. We prove that 
	\begin{equation}
		\label{decompUhPhiIJ}
		u_h = \sum_{(i,j) \in \mathcal{H}(\Omega)} \lambda_{i,j} \phi_{i,j}\,,	
	\end{equation}
	by showing that this equality holds on the interior of each mesh element $K \in \Omega_h$. Thus let us fix an element $K$ and write its Lagrange nodes as $\vec x_{i_1}, \ldots, \vec x_{i_R}$. For each $r \in \{1,\ldots,R\}$, let $j_r$ be such that $K \in \textup{st}(\vec x_{i_r};j_r)$. Let $y_n^r$ be a sequence of points in the interior of $K$ converging to $\vec x_{i_r}$. We claim that 
	\begin{equation}
		\label{limitUh}
		\lim_{n \to \infty} \phi_{i,j}(y_n^r) = \delta_{i,i_r} \delta_{j,j_r}\,, \quad 
		\lim_{n \to \infty} u_h(y^{r}_n) = \lambda_{i_r,j_r}\,.
	\end{equation}
	Only the second limit deserves attention. It is immediate if $K = K_{i_r,j_r}$\footnote{If $\vec x_{i_r}$ is an interior Lagrange node, that is, $\vec x_{i_r} \in \textup{int}(K)$, then one always has $K = K_{i_r,j_r}$. However, if $\vec x_{i_r,j_r}$ lies on $\partial K$, then $K_{i_r,j_r}$ can be any other mesh element in $\textup{st}(\vec {x}_{i_r};j_r)$.}. If $K$ shares a face $F \notin \mathcal{M}_\Gamma$ with $K_{i_r,j_r}$, then it is a consequence of Lemma \ref{lem:match}. Otherwise, by definition of $G(\vec{x}_i)$, one can find a face-connected (through faces not in $\Gamma$) path of mesh elements from $K$ to $K_{i_p,j_p}$, and the desired limit is established by repeating the previous argument for each pair of consecutive elements in this path.
	
	Having established the equalities in \eqref{limitUh}, we conclude that the polynomials defined on $K$ by each side of eq.~\eqref{decompUhPhiIJ} coincide at all the Lagrange points in $K$, and are both of degree $p$, thus they are equal on $K$ by unisolvance of $\mathcal{L}_p(K)$.
\end{proof}
For $u_h \in V^p(\Omega_h;\Gamma)$, we will denote by $u_h(\vec x_{i,j})$ the coefficient of $u_h$ on the split basis function $\phi_{i,j}$, so that 
\[u_h = \sum_{(i,j) \in \mathcal{I}} u_h(\vec x_{i,j})\phi_{i,j}\,.\]
We define the discrete scalar product 
\begin{equation*}
	\forall (u_h,v_h) \in V^p(\Omega_h;\Gamma) \times V^p(\Omega_h;\Gamma)\,, \quad [u_h,v_h]_{l^2} \isdef \sum_{(i,j) \in \mathcal{I}} u_h(\vec x_{i,j}) v_h(\vec x_{i,j})\,.
\end{equation*}
We identify $V^p(\Omega_h)$ with the subspace of $V^p(\Omega_h;\Gamma)$ defined by $u_h(\vec x_{i,j}) = u_h(\vec x_{i,j'})$ for all $1 \leq j,j' \leq q_i$ and let $\Psi_h^p(\Omega ; \Gamma)$ be the $[\cdot,\cdot]_{l^2}$ orthogonal complement of $V^p(\Omega_h)$ in $V^p(\Omega_h;\Gamma)$. Let
\[\widetilde{\mathcal{I}}\isdef \enstq{(i,j) \in \mathcal{I}}{q_i > 1 \textup{ and } 1 \leq j \leq q_i - 1}\,,\]
and for $(i,j) \in \widetilde{\mathcal{I}}$, define 
\[\psi_{i,j} \isdef \phi_{i,j} - \phi_{i,q_i}\,.\]
Using that $\{\phi_i\}_{1 \leq i \leq N}$ is a basis of $V^p(\Omega_h)$, together with the property
\[\phi_{i} = \sum_{j = 1}^{q_i} \phi_{i,j}\]
and \Cref{basisVh}, one deduces that $\{\psi_{i,j}\}_{(i,j) \in \widetilde{\mathcal{I}}}$ is a basis of $\Psi^p(\Omega_h; \Gamma)$.
\begin{corollary}
	There holds
	\begin{equation}
		V^p(\Omega_h;\Gamma) = V^p(\Omega_h) \oplus \Psi^p(\Omega_h; \Gamma)\,,
	\end{equation}
	and $\{\phi_i\}_{1 \leq i \leq N}$, $\{\psi_{i,j}\}_{(i,j) \in \widetilde{\mathcal{I}}}$ are bases of $V^p(\Omega_h)$ and $\Psi^p(\Omega_h; \Gamma)$, respectively.
\end{corollary}

\subsection{Dual basis of $V^p(\Omega_h;\Gamma)$}
\label{sec:dualBasis} 
With the primal basis 
$\{\phi_i\}_{1 \leq i \leq N} \cup \{\psi_{i,j}\}_{(i,j) \in \widetilde{\mathcal{I}}}$
at hand, we now set out to define a ``dual basis" $\{N_i\}_{1 \leq i \leq N} \cup \{N_{i,j}\}_{(i,j) \in \widetilde{\mathcal{I}}}$, consisting of linear forms on $N_i, N_{k,l} : H^1(\Omega \setminus \Gamma) \to \R$ such that, among other properties,
\[\begin{array}{ccccccc}
	N_i(\phi_{i'}) &=& \delta_{i,i'}\,, && N_i(\psi_{k',l'}) &=& 0\,,\\
	N_{k,l}(\phi_{i'}) &=& 0\,, && N_{k,l}(\psi_{k',l'}) &=& \delta_{k,k'}\delta_{l,l'}\,,\\
\end{array}\]
for all $1 \leq i,i' \leq N$ and $(k,l),(k',l') \in \widetilde{\mathcal{I}}$. 
As in Scott and Zhang's original construction  \cite{scott1990finite}, we will use Lagrange dual polynomials:
\begin{definition}[Lagrange dual polynomials]
	For any $\gamma_0$ shape-regular $n$-simplex $S \subset \R^d$ and any Lagrange node $\vec x \in \mathcal{L}_p(S)$, there exists a polynomial $\psi_{[S,\vec x]} \in \mathbb{P}_p$ such that 
	\begin{equation}
		\label{eq:defPsiSx}
		\int_{S} \psi_{[S,\vec x]}(y) P(y)\,d\mathcal{H}^n(y) = P(\vec x)
	\end{equation}
	for all $P \in \mathbb{P}_r$, and furthermore,  
	\[\norm{\psi_{[S,\vec x]}}_{L^2(F)}  \lesssim \textup{diam}(S)^{-n/2}\,. \]
\end{definition}

This is shown by considering the dual basis of the vector space $\mathbb{P}_r$ on the reference element, representing it under the form \eqref{eq:defPsiSx} (with $S$ replaced by the reference element) via the Riesz representation theorem, and mapping back to $S$, tracking the scaling of the $L^2$ norm accordingly, see \cite[Lemma 3.1]{scott1990finite} and \cite[eq. (3.3)]{camacho2015pointwise}. 

\begin{lemma}
	\label{lem:testFunctions}
	Let $u_h \in V^p(\Omega_h;\Gamma)$ and let $K \in \Omega_h$, $F \in \mathcal{F}(K)$ and $\vec x_i \in \mathcal{L}^p(F)$ a Lagrange node on $F$. Let $1 \leq j \leq q_i$ be such that $K \in \textup{st}(\vec x_i;j)$. Then
	\[\int_{F} \psi_{[F,\vec x_i]} \gamma_{[K,F]} u_h(x)\,d\mathcal{H}^{d-1}(x) = u_h(\vec x_{i,j})\,.\]
\end{lemma}
\begin{proof}
	This is immediate from the definition of $\phi_{i,j}$ and $\psi_{[F,\vec x_i]}$. 
\end{proof}

The next lemma deals with the construction of $\{N_i\}_{1 \leq i \leq N}$. 
\begin{lemma}
	\label{lem:singleBasis}
	For every $i \in \{1,\ldots,N\}$, there exists a linear form $N_i: H^1(\Omega \setminus \Gamma) \to \R$, such that the following properties hold:
	\begin{itemize}
		\item[(i)] $N_i(u_h) = 0$ for all $u_h \in \Psi^p(\Omega_h;\Gamma)$ 
		\item[(ii)] $N_i(\phi_{i'}) = \delta_{i,i'}$ for all $i,i' \in \{1,\ldots,N\}$. 
		\item[(iii)] If $u \in H^1_{0,\Gamma}(\Omega)$, then $N_i(u) = 0$ for all $i$ such that $\vec x_i \in \Gamma \cup \partial \Omega$.
		\item[(iv)] $N_i$ is a linear combination of terms of the form
		\[w \mapsto \int_{K} \psi_{[K,\vec x_i]} w(x)d\mathcal{H}^d(x)\]
		and
		\[w \mapsto \int_{F} \psi_{[F,\vec x_i]}(x)\gamma_{[K_F,F]} w(x)\,d \mathcal{H}^{d-1}(x)\,,\]
		where $K$ (resp. $F$) is a mesh element of $\Omega_h$ (resp. of $\mathcal{F}(\Omega_h)$) incident to $\vec x_i$ and $K_F \in \textup{st}(\vec x_i)$ is such that $F \in \mathcal{F}(K_F)$. The number $N$ of terms in this linear combination satisfies $N \lesssim 1$.
	\end{itemize}
	
\end{lemma}
\begin{proof}
	Let $i \in \{1,\ldots,N\}$ be such that $\vec x_i \notin \Gamma$. We then choose a ``control simplex" $\sigma_i = \sigma_{i,1}$ such that $\vec x_i \in \sigma_i$ as in \cite{scott1990finite}. That is, if $\vec x_i$ is an interior Lagrange node, then we take for $\sigma_i$ the unique mesh element $K_i \in \Omega_h$ such that $\vec x_i \in K_i$, and otherwise, we choose for $\sigma_i$ any mesh face in $\mathcal{F}(\Omega_h)$ containing $\vec x_i$, with the restriction that $F\subset \partial \Omega$ if $\vec x_i \in \partial\Omega$. We then put
	\[N_i(u) \isdef \int_{\sigma_i} \psi_{[\sigma_i,\vec x_i]}(x) u(x)\,d\mathcal{H}^{n_i}(x)\,,\]
	where $n_i \in \{d-1,d\}$ is the dimension of $\sigma_i$, and where we understand $u(x)$ as $\gamma_{[K_i,\sigma_i]} u(x)$ if $\sigma_i$ is a face of $K_i$. Note that (i) holds since, when $q_i = 1$, there is no $j$ such that $(i,j) \in \widetilde{\mathcal{I}}$.
	
	On the other hand if $\vec x_i \in \Gamma$, then according to \Cref{lem:simple}, we can choose $K_{i,1} \in \textup{st}(\vec x_i;1),\ldots,K_{i,q_i} \in  \textup{st}(\vec x_i;q_i)$ such that for all $j = 1,\ldots,q_j$, $K_{i,j}$ has a face $F_{i,j}$ in $\Gamma$ with $\vec x_i\in F_{i,j}$. We then set
	\[N_i(u) \isdef \frac{1}{q_i}\sum_{j = 1}^{q_i} \int_{F_{i,j}} \psi_{[F_{i,j},\vec x_i]}(x) \gamma_{[K_{i,j},F_{i,j}]}u(x)\,d\mathcal{H}^{d-1}(x)\,.\]
	One has $N_i(u_h) = \frac{1}{q_i}\sum_{j = 1} u_h(\vec x_{i,j})$ for all $u_h \in V^p(\Omega_h;\Gamma)$ by \Cref{lem:testFunctions}, which implies (i) and (ii). The property (iii) is also immediate due to the choice of the simplices $\sigma_i$ and $F_{i,j}$. Finally, \Cref{prop:solidangles} implies that the number of elements in $\textup{st}(\vec x_i)$ is $\lesssim 1$, which, in combination with the upper bound \eqref{eq:qmax} on $q_i$, establishes the claim about $N$ in (iv).
\end{proof}
We have used the following lemma. 
\begin{lemma}
	\label{lem:simple}
	For each $(i,j) \in \mathcal{I}$ such that $q_i > 1$, there exists an element $K^* \in \textup{st}(\vec x_{i};j)$ which possesses a face $F \subset \Gamma$ and such that $\vec x_i \in F$. 
\end{lemma}
\begin{proof}
	Let $K \in \textup{st}(\vec x_i;j)$ and let $K' \in \textup{st}(\vec x_i)$ such that $K' \notin \textup{st}(\vec x_i;j)$. Since $\Omega_h$ is a regular mesh, by \Cref{prop:faceconnectedness} there is a face-connected path $K = K_0,K_1,\ldots,K_Q = K'$ in $\textup{st}(\vec x)$. Let $q \in \{1,\ldots,Q\}$ be the first index such that $K_q \notin \textup{st}(\vec x_i;j)$. Then $K^* \isdef K_q$ satisfies the requirements. 
\end{proof}

We now deal with the construction of the second family of linear forms. The goal is to prove the following Lemma:
\begin{lemma}
	\label{lem:jumpBasis}
	There exists a family of linear forms $\{N_{i,j}\}_{(i,j) \in \widetilde{\mathcal{I}}}$ such that the following properties hold:
	\begin{itemize}
		\item[(i)] $N_{i,j}(u) = 0$ for all $u \in H^1(\Omega)$,
		\item[(ii)] $N_{i,j}(\psi_{i',j'}) = \delta_{i,i'}\delta_{j,j'}$ for all $(i,j),(i',j') \in \widetilde{\mathcal{I}}$.
		\item[(iii)] $N_{i,j}$ is a linear combination of terms of the form
		\[w \mapsto \int_{F} \psi_{[F,\vec x_i]}(x)\gamma_{[K,F]} w(x)\,d\mathcal{H}^{d-1}(x)\,,\]
		where $F$ is a face incident to $\vec x_i$ and $K \in \textup{st}(\vec x_i)$ is such that $F \in \mathcal{F}(K)$. The number $N$ of terms in this linear combination satisfies $N \lesssim 1$.
	\end{itemize}
\end{lemma}
Each $N_{i,j}$ will be constructed as a linear combination of ``bridge functions" that we define now. 
\begin{definition}[Bridge functions]
	Given $k,\ell \in \{1,\ldots,q_i\}$, we say that $\{k,\ell\}$ is a {\em bridge} around $\vec{x}_i$ if there exist two mesh elements $K_k \in \textup{st}(\vec x_i;k)$, $K_\ell \in \textup{st}(\vec x_i;\ell)$ sharing a face $F_{k\ell} \isdef K_k \cap K_\ell \subset \Gamma$. The set of bridges around $\vec{x}_i$ is denoted by $\mathcal{B}(i)$. 
	Given a bridge $\{k,\ell\} \in \mathcal{B}(i)$ with $k < \ell$, we define the {\em bridge function} $N_{i,\{k,\ell\}}$ by 
	\[N_{i,\{k,\ell\}}(u) \isdef \int_{F_{k\ell}} \psi_{[F_{k\ell},\vec x_i]}(x) \big(\gamma_{[K_k,F_{k\ell}]}  - \gamma_{[K_\ell,F_{k\ell}]} \big)u(x)\,d\mathcal{H}^{d-1}(x)\,.\] 
\end{definition}
Note that, by \Cref{lem:testFunctions}, for all $u_h \in V^p(\Omega_h;\Gamma)$, 
\[N_{i,\{k,\ell\}}(u_h) = u_h(\vec x_{i,k}) - u_h(\vec x_{i,\ell})\,.\]
We now fix $i$ such that $q_i > 1$ and consider the vector space \begin{equation*}
	F_i \isdef \textup{Span}(\{\psi_{i,j}\}_{1 \leq j \leq q_i - 1})\,.
\end{equation*}
Let $F_i^*$ be the set of linear forms $L_i: F_i \to \R$, and let $n_{i,\{k,\ell\}} \in F_i^*$ be the restriction of $N_{i,\{k,\ell\}}$ to $F_{i}$.  

\begin{lemma}
	\label{lem:bridge}
	The bridge functions around $\vec x_i$ span $F_i^*$, i.e.,
	\[F_i^* = \textup{Span}(\{n_{i,\{k,\ell\}}\}_{\{k,\ell\} \in \mathcal{B}(i)})\,.\] 
\end{lemma}
\begin{proof}
	By a classical result on linear forms (see e.g. \cite[Lemma 3.9]{rudin1991functional}), it suffices to show that if $u_h \in F_i$ satisfies 
	\begin{equation}
		\label{eq:condBridge}
		n_{i,\{k,\ell\}}(u_h) = 0 \quad \forall \{k,\ell\} \in \mathcal{B}(i)\,,
	\end{equation}
	then $u_h = 0$. We first note that $u_h \in F_i$ implies that $\sum_{j= 1}^{q_i} u_h(\vec x_{i,j}) = 0$
	since this property holds for every $\psi_{i,j}$. Hence to conclude it remains to show that $u_h(\vec x_{i,j}) = u_h(\vec x_{i,j'})$ for $1 \leq j,j' \leq q_i$. For this, introduce the graph ${G}^*(\vec{x}_i)$ defined by
	\begin{itemize}
		\item {\bf Nodes}: the numbers $1,\ldots,q_i$,
		\item {\bf Edges}: the bridges $\{k,k'\}$.
	\end{itemize} 
	The key point is that ${G}^*(\vec x_{i})$ is connected, because $\textup{st}(\vec x_i)$ is face-connected by \Cref{prop:faceconnectedness}. Eq.~\eqref{eq:condBridge} immediately implies $u_h(\vec x_{i,j}) = u_h(\vec x_{i,j'})$ if $\{j,j'\} \in \mathcal{B}(i)$, and the same follows for a general pair $\{j,j'\} \in \{1,\ldots,q_i\}$ by using a path from $j$ to $j'$ in ${G}^*(\vec x_i)$. This concludes the proof of the Lemma. 
\end{proof}
We denote by $\{{n}_{i,j}\}_{1 \leq j \leq q_i - 1}$ the basis of $F_i^*$ such that
$n_{i,j}(\psi_{i,j'}) = \delta_{j,j'}$.
\begin{lemma}
	For every $(i,j) \in \widetilde{\mathcal{I}}$, there exists $$\{k_1,\ell_1\},\ldots,\{k_{q_i-1},\ell_{q_i - 1}\} \in \mathcal{B}(i)\,, \qquad \lambda_{i,1}\,,\ldots\,,\lambda_{i,q_i - 1} \lesssim 1$$ 
	such that 
	\[n_{i,j} = \sum_{q = 1}^{q_i - 1} \lambda_{i,q} n_{i,\{k_q,\ell_q\}}\,.\]
\end{lemma}
\begin{proof}
	Let $Q = \max_{1 \leq i \leq N} (q_i - 1)$. Since $\textup{dim}(F_i^*) = \textup{dim}(F_i) = q_i - 1$, one can find a set of pairs $\{k_q,\ell_q\}_{1 \leq q \leq q_i - 1}$ such that $\{n_{i,\{k_q,\ell_q\}}\}_{1 \leq q \leq q_i - 1}$ is a basis of $F_i^*$. One can write $n_{i,j}$ in a unique way as
	\[n_{i,j} = \sum_{q = 1}^{q_i - 1} \lambda_{i,q} n_{i,\{k_q,\ell_q\}}\,. \]
	The coefficients $\lambda_{i,q}$ can be found by solving the linear system
	\[\mathbf{A} \begin{bmatrix}
		\lambda_{i,1}&
		\ldots&
		\lambda_{i,q_i - 1}
	\end{bmatrix}^T = \vec e_j\,,\]
	where $\vec e_j$ is the $j$-th vector of the canonical basis of $\R^{q_i - 1}$ and the matrix coefficients 
	\[\mathbf{A}_{p,p'} \isdef \psi_{i,p'}(\vec x_{i,k_q}) - \psi_{i,q'}(\vec x_{i,\ell_q})\]
	are integer between $-2$ and $2$. The set $\mathcal{S}_{Q}$ of invertible square matrices $\mathbf{A}$ of size at most $Q$ with coefficients in $\{-2,-1,0,1,2\}$ is finite. Therefore, we can write 
	\begin{equation}
		\label{eq:indepBound}
		\abs{\lambda_{i,q}} \leq \sup_{\mathbf{A} \in \mathcal{S}_{Q}}|||\mathbf{A}^{-1}|||_{\infty}\,, \quad \forall q \in \{1\,,\ldots\,,q_i - 1\}\,,
	\end{equation} 
	where, for an element $v$ of $\R^r$ and a square $r \times r$ matrix $M$, we have denoted 
	\[\norm{v}_{\infty} \isdef \max_{1 \leq q \leq r} \abs{v_q}\,, \quad ||| M |||_{\infty} \isdef \sup_{v \in R^{r}} \frac{\norm{M v}_{\infty}}{\norm{v}_{\infty}} \,.\] 
	Since $Q \lesssim 1$, the right-hand side of the inequality \eqref{eq:indepBound} is $\lesssim 1$.
\end{proof}
\begin{proof}[Proof of \Cref{lem:jumpBasis}]
	We define  
	\[N_{i,j} \isdef \sum_{q = 1}^{q_i - 1} \lambda_{i,q} N_{i,\{k_q,\ell_q\}}\,.\]
	The properties (i)-(iii) follow immediately by construction and the uniform bound on the number of elements in $\textup{st}(\vec x_i)$. 
\end{proof}

\subsection{Definition of $\Pi_h$ and proof of \Cref{thm:alg1,thm:cont1}} 
\label{sec:defPih}
We define 
\begin{equation}
	\label{defPih}
	\Pi_h u \isdef \sum_{i = 1}^N N_i(u)\, \phi_i\,\,\, + \hspace{-0.3cm} \sum_{(i,j) \in \widetilde{\mathcal{I}}} \hspace{-0.3cm} N_{i,j}(u) \,\psi_{i,j}\,.
\end{equation}
\Cref{thm:alg1} is a consequence of the properties of $N_i$ and $N_{i,j}$. The rest of this section is devoted to the proof of \Cref{thm:cont1}. 
\begin{lemma}[Bramble-Hilbert lemma]
	\label{lem:BH2}
	Let $p,d \geq 1$ be integers, $\gamma_0> 0$ a real number. Then for any $\gamma_0$-shape-regular $d$-simplex $K$ and $0 < t \leq p+1$, and $u \in H^t(K)$, there exists $p_K \in \mathbb{P}_p$ such that
	\begin{equation}
		\label{eq:BHlem}
		\norm{u - p_K}_{H^s(K)} \leq C(\gamma_0)h_K^{t-s} \abs{u}_{H^t(K)}
	\end{equation}
\end{lemma}
\begin{proof}
 	This is well-known, see the seminal work \cite{bramble1970estimation} (the proof for fractional Sobolev exponents is obtained by interpolation), and \cite{dupont1980polynomial}. 
\end{proof}
\begin{lemma}[Scaled trace theorem]
	\label{lem:scaledTrace}
	For every $\gamma_0 > 0$, $t \in (\frac12, 1]$, there exists $C(\gamma_0,t) > 0$ such that, if $K \subset \R^d$ is a $\gamma_0$-shape-regular $d$-simplex ($d \geq 2$) and $F \in \mathcal{F}(K)$, then for all $u \in H^t(K)$
	\[\norm{u}_{L^2(F)} \leq C(\gamma_0,t)\left(\textup{diam}(K)^{t - \frac{1}{2}} \abs{u}_{H^t(K)} + \textup{diam}(K)^{-\frac12} \|u\|_{L^2(K)}\right)\]
\end{lemma}
\begin{proof}
	This is shown by mapping to the reference element, using the trace inequality \eqref{eq:traceIneq} and pulling back, using the scaling porperties of the (semi-) norms $\|\cdot\|_{L^2(K)},\abs{\cdot}_{H^t(K)}$ and $\|\cdot\|_{L^2(F)}$. 
\end{proof}

Fix a mesh element $K \in \Omega_h$ and let $u \in H^1(\Omega \setminus \Gamma) \cap H^t_{\rm pw}(\omega_K)$ (as defined in \Cref{sec:mainResult1}). For every $K' \in \omega_K$, let $p_{K'}$ be a polynomial such that $u-p_K'$ satisfies the Bramble-Hilbert estimate \eqref{eq:BHlem} on $K'$.
\begin{lemma}
	\label{lem:aux}
	For any $v_h \in V^p(\Omega_h;\Gamma)$ such that $(v_h)_{|K} = p_K$, there holds  
	\[\norm{u - \Pi_h u}_{H^s(K)} \lesssim h^{t - s}|u|^2_{H^t(K)} +  h^{\frac32 - s}\sum_{\nu = 1}^{N_\nu} \abs{\int_{F_\nu} \psi_{[F_\nu;\vec y_\nu]}(x) \gamma_{[K_\nu;F_\nu]}(u - v_h )(x)\,d\mathcal{H}^{d-1}(x)}\]
	where $N_\nu \lesssim 1$, and for all $1 \leq \nu \leq N_\nu$, $\vec y_\nu \in \mathcal{L}_p(K)$, $K_\nu \in \textup{st}(\vec y_\nu)$ and $F_\nu \in \mathcal{F}(K_\nu)$.
\end{lemma}
\begin{proof}
	Let $\widetilde{N}_{i,j}$ be the linear forms such that 
	\[\Pi_h z = \sum_{(i,j) \in \mathcal{I}} \widetilde{N}_{i,j}(z) \phi_{i,j}\,.\]
	We can write
	\[\begin{split}
		\norm{u - \Pi_hu}_{H^s(K)} &\leq \norm{u - v_h}_{H^s(K)} + \norm{\Pi_h (u - v_h) }\,,\\
		&\lesssim h^{t - s} \abs{u}_{H^t} + \sum_{(i,j)} |\widetilde{N}_{i,j}(u - v_h)| \norm{\phi_{i}}_{H^s(K)}\\
		& \lesssim  h^{t - s} \abs{u}_{H^t} + h^{3/2-s} \sum_{(i,j)} |\widetilde{N}_{i,j}(u - v_h)|
	\end{split}\]
	where the sum runs over the pairs $(i,j) \in \mathcal{I}$ such that $K \in \textup{st}(\vec x_i;j)$. For all $(i,j) \in \mathcal{I}$ and $w \in H^1(\Omega \setminus \Gamma)$,   $\widetilde{N}_{i,j}(w)$ is a linear combination of $N_i(w)$ and $N_{i,j'}(w)$ for $1 \leq j' \leq q_i - 1$, which in turn are both defined via linear combinations of terms of the form
	\[\int_{F_\nu} \psi_{[F_\nu;\vec x_\nu]}(x) \gamma_{[K_\nu,F_\nu]} w(x)\,d\mathcal{H}^{d-1}(x) \] 
	with $\vec x_\nu$, $F_\nu$ and $K_\nu$ as in the statement of the Lemma, or terms of the form 
	\[\int_{K} \psi_{[K;\vec x_i]} w(x)\,dx\,.\]
	when $\vec x_i$ is an interior Lagrange node. In the latter case, we can write 
	\[\begin{split}
		\abs{\int_{K} \psi_{[K;\vec x_i]}} (u(x) - v_h(x))dx& = \abs{\int_{K} \psi_{[K;\vec x_i]}} (u(x) - p_K(x))dx\\
		&\leq \norm{\psi_{[K;\vec x_i]}}_{L^2(K)} \norm{u - p_K}_{L^2(K)} \\
		& \leq h^{-3/2} \abs{u}_{H^t}\,.
	\end{split} \]
	We conclude the proof using again that  $\textup{card}(\textup{st}(\vec x_i)) \lesssim 1$. 
\end{proof}
Define $v_h \in V^p(\Omega_h;\Gamma)$ by
\[v_h := \sum_{(i,j) }p_{K_{i,j}}(\vec x_i) \phi_{i,j}\]
where the sum ranges over the pairs $(i,j)$ such that $\vec x_i \in \mathcal{L}_p(K)$, and where the element $K_{i,j} \in \Omega_h$ is chosen freely in $\textup{st}(\vec x_i;j)$, with the only constraint that $K_{i,j} = K$ whenever possible, i.e. whenever $\textup{st}(\vec x_i;j) \ni K$. This constraint then ensures $(v_h)|_K = p_K$ by unisolvence, so that \Cref{lem:aux} applies to $v_h$. 

We now estimate a term of the form 
\[\int_{F_\nu} \psi_{[F_\nu;\vec y_\nu]}(x) \gamma_{[K_\nu,F_\nu]} (u(x) - v_h(x))\,d\mathcal{H}^{d-1}(x)\]
as defined in that Lemma. Let $(i_0,j_0)$ be the unique element of $\mathcal{I}$ determined by
\[\vec y_\nu = \vec x_{i_0} \quad \textup{and} \quad K_\nu \in \textup{st}(\vec x_{i_0};j_0)\]
Then for any $(i,j) \in \mathcal{I}$, one has
\[\begin{split}
	\int_{F_\nu} \psi_{[F_\nu;\vec y_\nu]}(x) \gamma_{[K_\nu,F_\nu]} \phi_{i,j}(x) d\mathcal{H}^{d-1}(x)&= \begin{cases}
	 \phi_{i}(\vec x_{i_0})& \textup{if } K_\nu \in \textup{st}(\vec x_{i};j),\\
	0 & \textup{otherwise,} 
\end{cases}\\[1em]
&= \delta_{i,i_0}\delta_{j,j_0}\,. 
\end{split}\]
Therefore, for any $u_h \in V^p(\Omega_h; \Gamma)$, 
\[\int_{F_\nu} \psi_{[F_\nu;\vec y_\nu]}(x) \gamma_{[K_\nu,F_\nu]} u_h(x)\,d\mathcal{H}^{d-1}(x) = u_h(\vec x_{i_0,j_0})\,.\]
Thus,
\[\begin{split}
	&\int_{F_\nu} \psi_{[F_\nu;\vec y_\nu]}(x) \gamma_{[K_\nu,F_\nu]}(u(x)-v_h(x))\,d\mathcal{H}^{d-1}(x) \\
	&\qquad \quad =  \int_{F_\nu} \psi_{[F_\nu;\vec y_\nu]}(x) \gamma_{[K_\nu,F_\nu]}(u(x)-p_{K_{i_0,j_0}}(x))\,d\mathcal{H}^{d-1}(x)\,.
\end{split}\]
This quantity can be estimated using Veeser's trick from \cite[Theorem 1]{veeser2016approximating} (see Eq.~(23)), stated as an independent result in \cite[Lemma 3.1]{camacho2015pointwise}. In our notation, this reads
\[\begin{split}
	&\int_{F_\nu} \psi_{[F_\nu;\vec y_\nu]}(x) \gamma_{[K_\nu,F_\nu]}(u(x)-p_{K_{i_0,j_0}}(x))\,d\mathcal{H}^{d-1}(x) \\
	&\quad\qquad =\int_{F_\nu} \psi_{[F_\nu;\vec y_\nu]}(x) \gamma_{[K_\nu,F_\nu]}(u(x)-p_{K_{\nu}}(x))\,d\mathcal{H}^{d-1}(x) + \,\, p_{K_\nu}(\vec x_{i_0}) - p_{K_{i_0,j_0}}(\vec x_{i_0}) \\
	&\quad\qquad = \int_{F_\nu} \psi_{[F_\nu;\vec y_\nu]}(x) \gamma_{[K_\nu,F_\nu]}(u(x)-p_{K_{\nu}}(x))\,d\mathcal{H}^{d-1}(x) + \sum_{l = 1}^{M-1}\,\, p_{K_l}(\vec x_{i_0}) - p_{K_{l+1}}(\vec x_{i_0})
\end{split}\]
using a telescoping sum, where $K_1,\ldots,K_M$  is a sequence of mesh elements in $\textup{st}(\vec x_{i_0};j_0)$ such that $K_\nu = K_1$, $K_{i,j} = K_M$, and $K_l$ and $K_{l+1}$ share a face $F_l$ not in $\Gamma$ containing $\vec x_i$ for every $l = 0\,,\ldots\,, M-1$. The first term is estimated via the scaled trace theorem (\Cref{lem:scaledTrace})
\[\begin{split}
	&\abs{\int_{F_\nu} \psi_{[F_\nu;\vec y_\nu]}(x) \gamma_{[K_\nu,F_\nu]}(u(x)-p_{K_{i,j}}(x))\,d\mathcal{H}^{d-1}(x)} \\
	&\qquad\qquad\qquad\leq \norm{\psi_{[F_\nu,\vec x_i]}}_{L^2(F_\nu)} \norm{\gamma_{[K_\nu,F_\nu]}(u-p_{K_{i,j}})}_{L^2(F_\nu)} \\
	&\qquad\qquad\qquad \lesssim  h^{-3/2} \norm{u - p_{K_{i,j}}}_{L^2(K_{i,j})} + h^{-3/2 + t} \abs{u - p_{K_{i,j}}}_{H^t}\\
	&\qquad\qquad\qquad\lesssim h^{-\frac32 +t}|u|_{H^t(K_{i,j})}\,.
\end{split}\]
Similarly, for all $l \in \{1,\ldots,M-1\}$, we write
\[\begin{split}
	p_{K_{l}}(\vec x_i) - p_{K_{l+1}}(\vec x_i) &= \int_{F_l} \psi_{[F_l,\vec x_i]} (p_{K_l}(x) - p_{K_{l+1}}(x))\,d\mathcal{H}^{d-1}(x) \\
	&= \int_{F_l} \psi_{[F_l,\vec x_i]}\gamma_{[K_{l+1},F_l]} (u-p_{K_{l+1}})(x) - \gamma_{[K_{l},F_l]}( u - p_{K_{l}})(x) \,d\mathcal{H}^{d-1}(x)
\end{split}\]
where we have used \Cref{lem:match}. Therefore, 
\[\abs{p_{K_l}(\vec x_i) - p_{K_{l+1}}(\vec x_i)} \lesssim  h^{-\frac32 + t}\abs{u}_{H^t(K_{i,j})}\]
In view of \Cref{lem:aux}, this establishes \Cref{thm:cont1}.

\bibliographystyle{amsplain}
\bibliography{biblioJSZ.bib}

\end{document}